\documentclass[11pt]{amsart}
\usepackage[utf8]{inputenc}
\newcommand\mtop{1in}
\newcommand\mbottom{1in}
\newcommand\mleft{1.2in}
\newcommand\mright{1.2in}
\usepackage{amssymb}
\usepackage{amsthm}
\usepackage{stmaryrd}
\usepackage{wasysym}
\usepackage{mathrsfs}
\usepackage[dvipsnames,svgnames,x11names]{xcolor}
\usepackage{hyperref}
\hypersetup{colorlinks=true,linkcolor=mblue,citecolor=mblue}
\usepackage{dsfont}
\usepackage[all]{xy}
\SelectTips{cm}{10}  
\usepackage{float}
\usepackage[usegeometry,paper=8.5in:11in]{typearea}
\providecommand{\mparwidth}{1in}
\providecommand{\mtop}{1in}
\providecommand{\mbottom}{1in}
\providecommand{\mleft}{1.2in}
\providecommand{\mright}{1.2in}
\usepackage[top = \mtop, bottom = \mbottom, left = \mleft, right=\mright, marginparwidth=\mparwidth]{geometry}

\usepackage{fancyhdr}
\usepackage{setspace}
\usepackage{mathtools}
\usepackage{scalefnt}
\usepackage{microtype}
\usepackage{pifont}
\usepackage{ifthen}
\usepackage{makecell}
\usepackage{marginnote}
\usepackage{longtable}

\usepackage{enumitem}
\setlist[enumerate]{itemsep=7pt}
\setlist[itemize,1]{leftmargin=4ex,topsep=0em}
\setlist[itemize]{itemsep=5pt}
\setlist[itemize,2]{label=$\circ$}
\setlist[itemize,3]{label={\scalefont{0.6}\color{gray}$\blacktriangleright$}}
\setlist[itemize,4]{label=$\ast$}
\setlist{nolistsep}

\DeclareMathOperator*{\holim}{holim}

\usepackage{mdframed}
\usepackage{tikz}
\usetikzlibrary{cd}
\usetikzlibrary{decorations.pathmorphing}
\usetikzlibrary{arrows, decorations.markings}

\AtBeginDocument{\storeareas\defaultareasettings}
\BeforeRestoreareas{\clearpage}

\usepackage[nolabel]{showlabels} 
\usepackage{soul}
\usepackage{comment}
\usepackage[final]{pdfpages}

\usepackage{xfp}

\usepackage{todonotes}
\definecolor{violet}{rgb}{0.56, 0.0, 1.0}

\newcommand{\MZ}{H\mathbb{Z}}

\newcommand{\ko}{\mathrm{ko}}
\newcommand{\KO}{\mathrm{KO}}
\newcommand{\KQ}{\mathrm{KQ}}

\newcommand{\ksp}{\mathrm{ksp}}

\KOMAoptions{twoside=false}
\begin{document}

\newcommand{\theoremnumstyle}{section}
\SelectTips{cm}{10}  

\allowdisplaybreaks 
\raggedbottom 

\newcommand*{\replacecommand}[1]{%
  \providecommand{#1}{}%
  \renewcommand{#1}%
}

\renewcommand{\arraystretch}{1.5}

\renewcommand{\top}{\mathrm{top}}

\renewcommand{\l}{\overset}
\newcommand{\into}{\hookrightarrow}
\newcommand{\onto}{\twoheadrightarrow}
\newcommand{\tto}{\longrightarrow}
\newcommand{\too}[1]{\l{#1}\to}
\newcommand{\ttoo}[1]{\l{#1}\longrightarrow}
\newcommand{\intoo}[1]{\l{#1}\into}
\newcommand{\ontoo}[1]{\l{#1}\onto}
\newcommand{\mapstoo}[1]{\l{#1}\mapsto}
\newcommand{\bto}{\leftarrow}
\newcommand{\btto}{\longleftarrow}
\newcommand{\btoo}[1]{\l{#1}\bto}
\newcommand{\bttoo}[1]{\l{#1}\longleftarrow}
\newcommand{\binto}{\hookleftarrow}
\newcommand{\bonto}{\twoheadleftarrow}
\newcommand{\bintoo}[1]{\l{#1}\binto}
\newcommand{\bontoo}[1]{\l{#1}\bonto}
\newcommand{\ointo}{\hspace{3pt}\text{\raisebox{-1.5pt}{$\overset{\circ}{\vphantom{}\smash{\text{\raisebox{1.5pt}{$\into$}}}}$}}\hspace{3pt}}
\newcommand{\lu}{\underset}
\newcommand{\bimplies}{\impliedby}
\newcommand{\ints}{\cap}
\newcommand{\intss}{\bigcap}
\newcommand{\union}{\cup}
\newcommand{\unions}{\bigcup}
\newcommand{\djunion}{\sqcup}
\newcommand{\djunions}{\bigsqcup}
\newcommand{\propersubset}{\subsetneq}
\newcommand{\propersupset}{\supsetneq}
\newcommand{\contains}{\supseteq}
\newcommand{\semidirect}{\rtimes}
\newcommand{\isom}{\cong}
\newcommand{\normal}{\triangleleft}
\replacecommand{\dsum}{\oplus}
\newcommand{\dsums}{\bigoplus}
\newcommand{\tensor}{\otimes}
\newcommand{\tensors}{\bigotimes}
\newcommand{\cotensor}{{\,\scriptstyle\square}}
\let\originalbar\bar
\renewcommand{\bar}[1]{{\overline{#1}}}
\newcommand{\rlim}{\mathop{\varinjlim}\limits}
\newcommand{\llim}{\mathop{\varprojlim}\limits}
\newcommand{\x}{\times}
\replacecommand{\st}{\hspace{2pt} : \hspace{2pt}} 
\newcommand{\vv}{\vspace{10pt}}
\newcommand{\til}{\widetilde}
\renewcommand{\hat}{\widehat}
\newcommand{\hhat}{\wedge}
\newcommand{\iy}{\infty}
\newcommand{\hteq}{\simeq}
\newcommand{\dd}[2]{\frac{\partial #1}{\partial #2}}
\newcommand{\sm}{\wedge} 
\newcommand{\noqed}{\renewcommand{\qedsymbol}{}}
\newcommand{\adjoint}{\dashv}
\newcommand{\wreath}{\wr}
\newcommand{\heart}{\heartsuit}

\newcommand{\bigast}{\mathop{\vphantom{\sum}\mathchoice%
  {\vcenter{\hbox{\huge *}}}
  {\vcenter{\hbox{\Large *}}}{*}{*}}\displaylimits}

\renewcommand{\dim}{\operatorname{dim}}
\newcommand{\diam}{\operatorname{diam}}
\newcommand{\coker}{\operatorname{coker}}
\newcommand{\im}{\operatorname{im}}
\newcommand{\disc}{\operatorname{disc}}
\newcommand{\Pic}{\operatorname{Pic}}
\newcommand{\Der}{\operatorname{Der}}
\newcommand{\ord}{\operatorname{ord}}
\newcommand{\nil}{\operatorname{nil}}
\newcommand{\rad}{\operatorname{rad}}
\newcommand{\ssum}{\operatorname{sum}}
\newcommand{\codim}{\operatorname{codim}}
\newcommand{\cchar}{\operatorname{char}}
\newcommand{\sspan}{\operatorname{span}}
\newcommand{\rank}{\operatorname{rank}}
\newcommand{\Aut}{\operatorname{Aut}}
\newcommand{\Out}{\operatorname{Out}}
\newcommand{\Div}{\operatorname{Div}}
\newcommand{\Gal}{\operatorname{Gal}}
\newcommand{\Hom}{\operatorname{Hom}}
\newcommand{\Mor}{\operatorname{Mor}}
\newcommand{\Vect}{\operatorname{Vect}}
\newcommand{\Fun}{\operatorname{Fun}}
\newcommand{\Iso}{\operatorname{Iso}}
\newcommand{\Map}{\operatorname{Map}}
\newcommand{\Ho}{\operatorname{Ho}}
\newcommand{\Mod}{\operatorname{Mod}}
\newcommand{\Tot}{\operatorname{Tot}}
\newcommand{\cofib}{\operatorname{cofib}}
\newcommand{\fib}{\operatorname{fib}}
\newcommand{\hocofib}{\operatorname{hocofib}}
\newcommand{\hofib}{\operatorname{hofib}}
\newcommand{\Maps}{\operatorname{Maps}}
\newcommand{\Sym}{\operatorname{Sym}}
\newcommand{\Diff}{\operatorname{Diff}}
\newcommand{\Tr}{\operatorname{Tr}}
\newcommand{\Frac}{\operatorname{Frac}}
\renewcommand{\Re}{\operatorname{Re}} 
\renewcommand{\Im}{\operatorname{Im}}
\newcommand{\gr}{\operatorname{gr}}
\newcommand{\tr}{\operatorname{tr}}
\newcommand{\End}{\operatorname{End}}
\newcommand{\Mat}{\operatorname{Mat}}
\newcommand{\Proj}{\operatorname{Proj}}
\newcommand{\Th}{\operatorname{Thom}}
\newcommand{\Thom}{\operatorname{Thom}}
\newcommand{\Spec}{\operatorname{Spec}}
\newcommand{\Ext}{\operatorname{Ext}}
\newcommand{\Cotor}{\operatorname{Cotor}}
\newcommand{\Tor}{\operatorname{Tor}}
\newcommand{\vol}{\operatorname{vol}}

\newcommand{\Set}{\operatorname{Set}}
\newcommand{\Top}{\operatorname{Top}}
\newcommand{\Fin}{\operatorname{Fin}}
\newcommand{\Spaces}{\operatorname{Spaces}}
\newcommand{\Sp}{\operatorname{Sp}}
\newcommand{\Spectra}{\operatorname{Spectra}}
\newcommand{\Spt}{\operatorname{Spt}}
\newcommand{\Comod}{\operatorname{Comod}}
\newcommand{\Spf}{\operatorname{Spf}}
\newcommand{\tmf}{\mathit{tmf}}
\newcommand{\Tmf}{\mathit{Tmf}}
\newcommand{\TMF}{\mathit{TMF}}
\newcommand{\Null}{\operatorname{Null}}
\newcommand{\Fil}{\operatorname{Fil}}
\newcommand{\Sq}{\operatorname{Sq}}
\newcommand{\Stable}{\operatorname{Stable}}
\newcommand{\Poly}{\operatorname{Poly}}
\newcommand{\Cat}{\operatorname{Cat}}
\newcommand{\Orb}{\operatorname{Orb}}
\newcommand{\Exc}{\operatorname{Exc}}
\newcommand{\Part}{\operatorname{Part}}
\newcommand{\Comm}{\operatorname{Comm}}
\newcommand{\Res}{\operatorname{Res}}
\newcommand{\Thick}{\operatorname{Thick}}
\newcommand{\red}{{\operatorname{red}}}

\newcommand{\Sm}{\operatorname{Sm}}
\newcommand{\Var}{\operatorname{Var}}
\newcommand{\Frob}{\operatorname{Frob}}
\newcommand{\Rep}{\operatorname{Rep}}
\newcommand{\Ch}{\operatorname{Ch}}
\newcommand{\Shv}{\operatorname{Shv}}
\newcommand{\Corr}{\operatorname{Corr}}
\newcommand{\Span}{\operatorname{Span}}
\newcommand{\Sch}{\operatorname{Sch}}
\newcommand{\ev}{\operatorname{ev}}
\newcommand{\Homog}{\operatorname{Homog}}
\newcommand{\conn}{\operatorname{conn}}
\newcommand{\type}{\operatorname{type}}
\newcommand{\num}{\operatorname{num}}
\newcommand{\Aff}{\operatorname{Aff}}
\newcommand{\Psh}{\operatorname{Psh}}
\newcommand{\sk}{\operatorname{sk}}
\newcommand{\cosk}{\operatorname{cosk}}
\newcommand{\Cart}{\operatorname{Cart}}

\newcommand{\Br}{\operatorname{Br}}
\newcommand{\BW}{\operatorname{BW}}
\newcommand{\Cl}{\operatorname{Cl}}
\newcommand{\Conf}{\operatorname{Conf}}
\newcommand{\Alg}{\operatorname{Alg}}
\newcommand{\CAlg}{\operatorname{CAlg}}
\newcommand{\Lie}{\operatorname{Lie}}
\newcommand{\Coalg}{\operatorname{Coalg}}
\newcommand{\Ab}{\operatorname{Ab}}
\newcommand{\Ind}{\operatorname{Ind}}
\newcommand{\ind}{\operatorname{ind}}
\newcommand{\Fix}{\operatorname{Fix}}
\newcommand{\ho}{\operatorname{ho}}
\newcommand{\coeq}{\operatorname{coeq}}
\newcommand{\CMon}{\operatorname{CMon}}
\newcommand{\Sing}{\operatorname{Sing}}
\newcommand{\Inj}{\operatorname{Inj}}
\newcommand{\StMod}{\operatorname{StMod}}
\newcommand{\Loc}{\operatorname{Loc}}
\newcommand{\Free}{\operatorname{Free}}
\newcommand{\Art}{\operatorname{Art}}
\newcommand{\Gpd}{\operatorname{Gpd}}
\newcommand{\Def}{\operatorname{Def}}
\newcommand{\Hyp}{\operatorname{Hyp}}
\newcommand{\Pre}{\operatorname{Pre}}
\newcommand{\Lat}{\operatorname{Lat}}
\newcommand{\Coords}{\operatorname{Coords}}
\newcommand{\cone}{\operatorname{cone}}
\newcommand{\Spc}{\operatorname{Spc}}
\newcommand{\QCoh}{\operatorname{QCoh}}
\newcommand{\height}{\operatorname{ht}}
\newcommand{\Sub}{\operatorname{Sub}}
\newcommand{\Cone}{\operatorname{Cone}}
\newcommand{\Cocone}{\operatorname{Cocone}}
\newcommand{\Ran}{\operatorname{Ran}}
\newcommand{\Lan}{\operatorname{Lan}}
\newcommand{\LieAlg}{\operatorname{LieAlg}}
\newcommand{\Com}{\operatorname{Com}}
\newcommand{\CoAlg}{\operatorname{CoAlg}}
\newcommand{\Prim}{\operatorname{Prim}}
\newcommand{\Coh}{\operatorname{Coh}}
\newcommand{\FormalGrp}{\operatorname{FormalGrp}}
\newcommand{\Fact}{\operatorname{Fact}}
\renewcommand{\Bar}{\operatorname{Bar}}
\newcommand{\Cobar}{\operatorname{Cobar}}
\newcommand{\Ad}{\operatorname{Ad}}
\newcommand{\Moduli}{\operatorname{Moduli}}
\newcommand{\dgla}{\operatorname{dgla}}
\newcommand{\obl}{\operatorname{obl}}
\newcommand{\IndCoh}{\operatorname{IndCoh}}
\newcommand{\Cocomm}{\operatorname{Cocomm}}
\newcommand{\PreStk}{\operatorname{PreStk}}
\newcommand{\FormGrp}{\operatorname{FormGrp}}
\newcommand{\FormMod}{\operatorname{FormMod}}
\newcommand{\Grp}{\operatorname{Grp}}
\newcommand{\CommAlg}{\operatorname{CommAlg}}
\newcommand{\CoComm}{\operatorname{CoComm}}
\newcommand{\IndSch}{\operatorname{IndSch}}
\newcommand{\Dist}{\operatorname{Dist}}
\newcommand{\Triv}{\operatorname{Triv}}
\newcommand{\Oper}{\operatorname{Oper}}
\newcommand{\Bij}{\operatorname{Bij}}
\newcommand{\Syl}{\operatorname{Syl}}
\newcommand{\Inn}{\operatorname{Inn}}
\newcommand{\Emb}{\operatorname{Emb}}
\newcommand{\Gr}{\operatorname{Gr}}
\newcommand{\CRing}{\operatorname{CRing}}
\newcommand{\sSet}{\operatorname{sSet}}
\newcommand{\et}{\text{\'et}}
\newcommand{\Sh}{\operatorname{Sh}}
\newcommand{\Nil}{\operatorname{Nil}}
\newcommand{\Cech}{\v Cech}

\newcommand{\A}{\mathbb{A}}
\replacecommand{\C}{\mathbb{C}}
\newcommand{\CP}{\mathbb{C}\mathrm{P}}
\newcommand{\E}{\mathbb{E}}
\newcommand{\F}{\mathbb{F}}
\replacecommand{\G}{\mathbb{G}}
\renewcommand{\H}{\mathbb{H}}
\newcommand{\K}{\mathbb{K}}
\newcommand{\M}{\mathbb{M}}
\newcommand{\N}{\mathbb{N}}

\newcommand{\bbO}{\mathcal{O}}

\renewcommand{\P}{\mathbb{P}}
\newcommand{\Q}{\mathbb{Q}}
\newcommand{\R}{\mathbb{R}}
\newcommand{\RP}{\mathbb{R}\mathrm{P}}
\newcommand{\V}{\vee}
\newcommand{\T}{\mathbb{T}}
\providecommand{\U}{\mathscr{U}}
\newcommand{\Z}{\mathbb{Z}}
\renewcommand{\k}{\Bbbk}
\newcommand{\g}{\mathfrak{g}}
\newcommand{\m}{\mathfrak{m}}
\newcommand{\n}{\mathfrak{n}}
\newcommand{\p}{\mathfrak{p}}
\newcommand{\q}{\mathfrak{q}}
\renewcommand{\t}{\mathfrak{t}}

\newcommand{\pa}[1]{\left( {#1} \right)}
\newcommand{\br}[1]{\left[ {#1} \right]}
\newcommand{\cu}[1]{\left\{ {#1} \right\}}
\newcommand{\ab}[1]{\left| {#1} \right|}
\newcommand{\an}[1]{\left\langle {#1}\right\rangle}
\newcommand{\fl}[1]{\left\lfloor {#1}\right\rfloor}
\newcommand{\ceil}[1]{\left\lceil {#1}\right\rceil}
\newcommand{\tf}[1]{{\textstyle{#1}}}
\newcommand{\patf}[1]{\pa{\textstyle{#1}}}

\renewcommand{\mp}{\ \raisebox{5pt}{\text{\rotatebox{180}{$\pm$}}}\ }
\renewcommand{\d}[1]{\ss \mathrm{d}#1}
\newcommand{\imod}{\hspace{-7pt}\pmod}

\renewcommand{\epsilon}{\varepsilon}
\renewcommand{\phi}{{\mathchoice{\raisebox{2pt}{\ensuremath\varphi}}{\raisebox{2pt}{\!\! \ensuremath\varphi}}{\raisebox{1pt}{\scriptsize$\varphi$}}{\varphi}}}
\newcommand{\ph}{{\color{white}.\!}}
\newcommand{\tspacer}{{\ensuremath{\color{white}\Big|\!}}}
\newcommand{\chii}{\raisebox{2pt}{\ensuremath\chi}}

\let\originalchi=\chi
\renewcommand{\chi}{{\!{\mathchoice{\raisebox{2pt}{
$\originalchi$}}{\!\raisebox{2pt}{
$\originalchi$}}{\raisebox{1pt}{\scriptsize$\originalchi$}}{\originalchi}}}}

\let\originalforall=\forall
\renewcommand{\forall}{\ \originalforall}

\let\originalexists=\exists
\renewcommand{\exists}{\ \originalexists}

\let\realcheck\check
\newcommand{\vH}{\realcheck{H}}

\newenvironment{qu}[2]
{\begin{list}{}
	  {\setlength\leftmargin{#1}
	  \setlength\rightmargin{#2}}
	  \item[]\footnotesize}
		  {\end{list}}

\newenvironment{titleblock}
{\begin{mdframed}[linecolor=black!20,backgroundcolor=black!15]\begin{center}}
{\end{center}\end{mdframed}}

\newenvironment{shadedblock}[1][5in]
{\bigskip\begin{mdframed}[align=center,userdefinedwidth=#1,linecolor=white,backgroundcolor=black!5]}{\end{mdframed}}

\newenvironment{shadedtitleblock}[2][5in]
{\begin{mdframed}[align=center,userdefinedwidth=#1,linecolor=white,backgroundcolor=black!15]\sc #2\end{mdframed}\begin{mdframed}[align=center,userdefinedwidth=#1,linecolor=white,backgroundcolor=black!5]}{\end{mdframed}}

\newcommand{\shadedheader}[1]{\vspace{15pt}\begin{mdframed}[linecolor=black!20,backgroundcolor=black!5]\sc #1\end{mdframed}\vspace{15pt}}


\newcommand{\itext}{\shortintertext} 
\makeatletter
\@ifundefined{resetu}{ 
	\renewcommand{\u}{\underbracket[0.7pt]} 
}
\makeatother

\makeatletter
\@ifundefined{mathds}{
	\newcommand{\Id}{Id}
	}{
	\newcommand{\Id}{\mathds{1}} 
	}
\@ifundefined{sethlcolor}{
	\newcommand{\fixmehl}[2]{\underline{#1}\marginpar{\raggedright\smaller\smaller #2}}
	}{\@ifundefined{marginnote}{
		\newcommand{\highlight}[1]{\ifmmode{\text{\sethlcolor{llgray}\hl{$#1$}}}\else{\sethlcolor{llgray}\hl{#1}}\fi}
		\newcommand{\fixmehl}[2]{\highlight{#1}\marginpar{\raggedright\smaller\smaller #2}}
		}{
		\newcommand{\highlight}[1]{\ifmmode{\text{\sethlcolor{llgray}\hl{$#1$}}}\else{\sethlcolor{llgray}\hl{#1}}\fi}
		\newcommand{\fixmehl}[2]{\marginnote{\smaller \smaller\color{Maroon} #2}{\highlight{#1}}}
		}
	}
\@ifundefined{color}{}{
	\definecolor{darkgreen}{RGB}{0,70,0}
	\definecolor{dgreen}{RGB}{0,100,0}
	\definecolor{purple}{RGB}{120,00,120}
	\definecolor{gray}{RGB}{100,100,100}
	\definecolor{mgreen}{RGB}{0,150,0}
	\definecolor{dgreen}{RGB}{0,100,0}
	\definecolor{llgray}{RGB}{230,230,230}
	\definecolor{lgreen}{RGB}{100,200,100}
	\definecolor{mgray}{RGB}{150,150,150}
	\definecolor{lgray}{RGB}{190,190,190}
	\definecolor{maroon}{RGB}{150,0,0}
	\definecolor{lblue}{RGB}{120,170,200}
	\definecolor{mblue}{RGB}{65,105,225}
	\definecolor{dblue}{RGB}{0,56,111}
	\definecolor{orange}{RGB}{255,165,0}
	\definecolor{brown}{RGB}{177,84,15}
	\definecolor{rose}{RGB}{135,0,52}
	\definecolor{gold}{RGB}{177,146,87}
	\definecolor{dred}{RGB}{135,19,19}
	\definecolor{mred}{RGB}{194,28,28}
	\newcommand{\edit}[1]{{\it{\color{gray}#1}}}
	\newcommand{\fixme}[1]{{\color{maroon}\it{#1}}}
	\newcommand{\citeme}[2][\!\!]{{\color{orange}[#2~\textit{#1}]}}
	\newcommand{\later}[1]{{\color{dgreen}#1}}
	\newcommand{\corr}[1]{{\color{red}\itshape #1}}
	\newcommand{\question}[1]{\itshape{\color{blue}#1}\upshape}
}
\@ifundefined{substack}{}{
    \newcommand{\attop}[1]{{\let\textstyle\scriptstyle\let\scriptstyle\scriptscriptstyle\substack{#1}}}
    \renewcommand{\atop}[1]{{\let\scriptstyle\textstyle\let\scriptscriptstyle\scriptstyle\substack{#1}}}
}
\makeatother

\newcommand{\tabentry}[1]{\renewcommand{\arraystretch}{1}\begin{tabular}{c}#1\end{tabular}}

\newcommand{\margin}[1]{\marginpar{\raggedright \scalefont{0.7}#1}} 

\newcommand{\pullback}{\ar@{}[rd]|<<{\text{\pigpenfont A}}}
\newcommand{\pushout}{\ar@{}[rd]|>>{\text{\pigpenfont I}}}

\newcommand{\longleftrightarrows}{\xymatrix@1@C=16pt{
\ar@<0.4ex>[r] & \ar@<0.4ex>[l]
}}
\newcommand{\longrightrightarrows}{\xymatrix@1@C=16pt{
\ar@<0.4ex>[r]\ar@<-0.4ex>[r] & 
}}
\newcommand{\mapstto}{\,\xymatrix@1@C=16pt{
\ar@{|->}[r] & 
}\,}
\newcommand{\mapsttoo}[1]{\xymatrix@1@C=16pt{
\ar@{|->}[r]^-{#1} & 
}}
\newcommand{\rightrightrightarrows}{\xymatrix@1@C=16pt{
\ar[r]\ar@<0.8ex>[r]\ar@<-0.8ex>[r] & 
}}
\newcommand{\longleftleftarrows}{\xymatrix@C=16pt{
 & \ar@<0.4ex>[l]\ar@<-0.4ex>[l]
}}
\newcommand{\leftleftleftarrows}{\xymatrix@1@C=16pt{
 & \ar[l]\ar@<0.8ex>[l]\ar@<-0.8ex>[l]
}}
\newcommand{\leftleftleftleftarrows}{\xymatrix@1@C=16pt{
 & \ar@<0.8ex>[l]\ar@<0.3ex>[l]\ar@<-0.3ex>[l]\ar@<-0.8ex>[l]
}}
\newcommand{\lcircle}{\ar@(ul,dl)} 
\newcommand{\rcircle}{\ar@(ur,dr)}
\newcommand{\intto}{\ \xymatrix@1@C=16pt{
\ar@{^(->}[r] & 
}}


\newtheoremstyle{gloss}{\topsep}{\topsep}{}{0pt}{\bfseries}{}{\newline}{\newline
*{\bf #3} }
\theoremstyle{gloss}
\newtheorem*{defstar}{Definition}

\newtheoremstyle{newplain}{20pt}{0pt}{\it}{0pt}{\bfseries}{.}{1ex}{}
\theoremstyle{newplain}

\ifthenelse{\isundefined\theoremnumstyle}
	{\newtheorem{theorem}{Theorem}[section] 
	\numberwithin{equation}{section}} 
	{\ifthenelse{\equal\theoremnumstyle{}}
		{\newtheorem{theorem}{Theorem}}
		{\newtheorem{theorem}{Theorem}[section]
		\numberwithin{equation}{section}
		}
	}

\newtheorem{corollary}[theorem]{Corollary}
\newtheorem{claim}[theorem]{Claim}
\newtheorem{lemma}[theorem]{Lemma}
\newtheorem{proposition}[theorem]{Proposition}
\newtheorem{fact}[theorem]{Fact}

\newtheoremstyle{newtextthm}{20pt}{0pt}{}{0pt}{\bfseries}{.}{1ex}{}
\theoremstyle{newtextthm}
\newtheorem{definition}[theorem]{Definition}
\newtheorem{example}[theorem]{Example}
\newtheorem{problem}[theorem]{Problem}
\newtheorem{remark}[theorem]{Remark}
\newtheorem{notation}[theorem]{Notation}

\newtheorem*{theoremstar}{Theorem}
\newtheorem*{lemmastar}{Lemma}
\newtheorem*{corstar}{Corollary}
\newtheorem*{corollarystar}{Corollary}
\newtheorem*{propositionstar}{Proposition}
\newtheorem*{claimstar}{Claim}
\newtheorem*{examplestar}{Example}

\newcommand{\argforrandom}{}
\theoremstyle{newtextthm}
\newtheorem{helperforrandom}[theorem]{\argforrandom}
\newtheorem*{helperforrandomstar}{\argforrandom}
\newenvironment{random}[1]{\renewcommand{\argforrandom}{#1}\begin{helperforrandom}}{\end{helperforrandom}}
\newenvironment{randomstar}[1]{\renewcommand{\argforrandom}{#1}\begin{helperforrandomstar}}{\end{helperforrandomstar}}

\newenvironment{exercise}[1]{\hspace{1pt}\nn \large {\sc #1.}\hv \normalsize
\vspace{10pt}\\ }{} 
\newcommand{\subthing}[1]{\hv\large(#1)\hv\hv \normalsize }

\newcommand{\itemref}[1]{(\ref{#1})}

\renewcommand{\showlabelfont}{\tiny\tt\color{mgreen}}
\newcommand{\val}{\operatorname{val}}
\newcommand{\jchartpath}{charts/charts_j/output/slicecharts_consolidated}
\newcommand{\jmockupchartpath}{charts/charts_j/d1-mockup.pdf}

\newcommand{\h}{\mathsf{h}}

\title{$\R$-motivic $v_1$-periodic homotopy}

\author{Eva Belmont}
\address{Department of Mathematics,
University of California San Diego,
La Jolla, CA 92093, USA}
\email{ebelmont@ucsd.edu}

\author{Daniel C. Isaksen}
\address{
Department of Mathematics, 
Wayne State University, 
Detroit, MI 48202, USA}
\email{isaksen@wayne.edu}

\author{Hana Jia Kong}
\address{School of Mathematics,
Institute for Advanced Study,
Princeton, NJ 08540, USA}
\email{hana.jia.kong@gmail.com}

\thanks{The second author was supported by NSF grant
DMS-1904241. The third author was supported by the National Science Foundation under Grant No. DMS-1926686.
This manuscript answers a question posed by Mark Behrens
to the second author in 2010 at the 
Conference on Homotopy Theory and Derived Algebraic Geometry, Fields
Institute, Toronto, Ontario, Canada.
The authors thank William Balderrama, Robert Bruner, and John Rognes for helpful discussions.
}



\begin{abstract}
We compute the $v_1$-periodic
$\R$-motivic stable homotopy groups.
The main tool is the effective slice spectral sequence.
Along the way, we also analyze $\C$-motivic and $\eta$-periodic
$v_1$-periodic homotopy from the same perspective.
\end{abstract}

\maketitle

\section{Introduction}
The computation of the 
stable homotopy groups of spheres is a difficult but central
problem of stable homotopy theory.  There is much that we do not know
about stable homotopy.  However, the $v_1$-periodic stable homotopy groups
(also known for historical reasons as the image of $j$) are 
completely
understood, and they have interesting number-theoretic properties.

The goal of this article is to explore $v_1$-periodic stable
homotopy in the $\R$-motivic context.  This choice of ground field
represents a middle ground between the well-understood $\C$-motivic
situation and the much more difficult situation of an arbitrary
field, in which arithmetic necessarily enters into the picture.

From our perspective, the field $\R$ introduces just one piece of
arithmetic: the failure of $-1$ to have a square root.  This leads to
complications in $\R$-motivic homotopical computations, but they
can be managed with care and attention to detail.

Classically, $v_1$-periodic homotopy is detected by the
connective spectrum $j^\top$, which is defined to be the fiber of a map
$\ko^\top \ttoo{\psi^3-1} \Sigma^4 \ksp^\top$, where 
$\ko^\top$ is the connective real $K$-theory spectrum, 
$\ksp^\top$ is the connective symplectic $K$-theory spectrum, and
$\psi^3$ is an Adams operation. (The ``$\top$" superscripts
indicate that we are discussing the classical context here,
rather than the motivic context.)

In fact, $\ko^\top$ itself is the more natural target for the map
$\psi^3 -1$.  However, the fiber of
$\ko^\top \ttoo{\psi^3-1} \ko^\top$ has a minor defect.
It has some additional homotopy classes in stems $-1$, $0$, and $1$
that do not correspond to homotopy classes for the sphere spectrum.
In other words, the map from $S^0$ to this fiber is not surjective
in homotopy.
If we change the target of $\psi^3-1$ from $\ko^\top$ to its connective
cover $\Sigma^4 \ksp^\top$, then this problem disappears, and the map
from $S^0$ to the fiber is onto in homotopy.

It is possible to mimic these constructions in motivic
stable homotopy theory \cite{BH20}.  
At the prime $2$,
one can define the motivic
connective spectrum $j$ to be the fiber of
a map $\ko \ttoo{\psi^3-1} \Sigma^{4,2} \ksp$, where 
$\ko$ is the very effective connective Hermitian $K$-theory spectrum,
$\ksp$ is defined in terms of very effective covers of $\ko$,
and $\psi^3$ is a motivic lift of an Adams operation.

However, from a computational perspective, this definition of $j$
introduces apparently unnecessary complications.
It is possible to compute the homotopy of $\R$-motivic $j$ using
the techniques that appear later in this manuscript.  
However, the computation is slightly messy, involving some exceptional
differentials and exceptional hidden extensions in low dimensions.  
In any case,
the homotopy of the $\R$-motivic sphere does not
surject onto the homotopy of $\R$-motivic $j$.  In other words,
the main rationale for using $\ksp$ in the first place does not apply
in the motivic situation.

On the other hand, the computation of the homotopy of
the $\R$-motivic fiber of $\ko \ttoo{\psi^3-1} \ko$ is much
cleaner.  Moreover, it tells us just as much about $v_1$-periodic
$\R$-motivic homotopy as $j$.  In other words, it has all of the
computational advantages of $j$, while avoiding some unfortunate complications.

Consequently, in this manuscript, we will be solely concerned with
the fiber of $\ko \ttoo{\psi^3-1} \ko$.  We use the notation $L$
for this fiber in order to avoid confusion with the traditional
meaning of $j$.
The symbol $L$ is meant to draw a connection to the classical
$K(1)$-local sphere $L_{K(1)} S^0$, which is the fiber of
$\KO^{\top} \ttoo{\psi^3-1} \KO^{\top}$.

However, the homotopy of the $\R$-motivic spheres does
not surject onto the homotopy of $\R$-motivic $L$.
It is possible that we may have not yet constructed the
``correct" motivic version of the classical connective
spectrum $j^\top$.
These considerations raise questions about vector bundles
and the motivic Adams conjecture.  We make no attempt to study these
more geometric issues.

We claim to compute the $v_1$-periodic $\R$-motivic
stable homotopy groups, but this claim deserves some clarification.
We do not use an intrinsic definition of $v_1$-periodic
$\R$-motivic homotopy, although such a definition could probably
be formulated in terms of the motivic $K(1)$-local sphere.
However, the theory of motivic $K(1)$-localization
has not yet been suitably developed.

Rather, we merely compute the homotopy of $L$, and we observe that
it detects large-scale structure in the stable homotopy
of the $\R$-motivic sphere, which was described in a range in
\cite{R-paper}.
In other words, we have a practical description of
$\R$-motivic $v_1$-periodic homotopy, not a theoretical one.

The careful reader may object that our approach with effective
spectral sequences is long-winded and unnecessarily complicated.
In fact, the homotopy of $L$ could be determined by direct analysis
of the long exact sequence associated to the defining fiber sequence
for $L$.  However, there is a disadvantage in this direct approach. 
We find that the effective filtration is useful additional information
about the homotopy of $L$ that helps us understand the computation.
The effective filtration is part of the ``higher structure" of the
homotopy of $L$.  
While we have no immediate uses for this higher structure,
we know from experience that it inevitably becomes important
in deeper homotopical analyses.

\subsection{Charts}

We provide a series of charts that display the effective spectral
sequences for $\ko$ and $L$, as well as their $\C$-motivic counterparts.
We consider these charts
to be the central achievement of this manuscript.  We encourage 
the reader to rely heavily on the charts.  In a sense, they provide
an illustrated guide to our computations.

Caution must be
exercised in the comparison to \cite{R-paper} since the Adams filtrations
and effective filtrations are different. 
As in \cite{R-paper}, our charts consider each coweight separately;
we have found that this is a practical way of studying
$\R$-motivic homotopy groups.
Periodicity by $\tau^4$ 
(which is not a permanent cycle, 
but should be thought of as a periodicity operator in coweight 4)
allows us to give a fairly compact depiction of the homotopy of $L$
in coweights congruent to 0, 1, and 2 modulo 4; see
Figures \ref{fig:L-Einfty:0}, \ref{fig:L-Einfty:1},
and \ref{fig:L-Einfty:2}.

The homotopy of $L$ in coweights congruent to 3 modulo 4 is much
more interesting but harder to describe.  
See Figures \ref{fig:L-Einfty:3mod8} and \ref{fig:L-Einfty:7mod16}.

\subsection{Completions}

We are computing exclusively in the $2$-complete context.  
This simplifies all
questions surrounding convergence of spectral sequences.  
Also, the final computational $2$-complete answers are easier to state 
than their $2$-localized or integral counterparts.

We generally omit completions from our notation for brevity.
For example, we write $\Z$ for the $2$-adic integers, and we write
$\KO$ for the $2$-completed $\R$-motivic Hermitian
$K$-theory spectrum.

Section \ref{subsctn:convergence} discusses these topics in
slightly more detail.

\subsection{Regarding the element $2$}
When passing from the effective $E_\infty$-page to stable homotopy groups,
one must choose homotopy elements that are represented by
each element of the $E_\infty$-page.  For the element $2$ in 
the $E_\infty$-page, there is more than one choice in
$\pi_{0,0}$ because of the presence of elements in the $E_\infty$-page 
in higher effective filtration.

From the perspective of abelian groups, the element $2 = 1 + 1$ is the obvious choice of homotopy element.
However, there is another element $\h$, also detected by $2$
in the effective spectral sequence, that turns out to be a much
more convenient choice.  The difference between $\h$ and $2$
in homotopy is detected by the element $\rho h_1$ in higher
filtration (to be discussed later).
Experience has shown that 
the motivic stable homotopy groups are
easier to describe in terms of $\h$ than in terms of $2$.  
{For example,
we have the relations $\h \rho = 0$ and $\h \eta = 0$,
where $\rho$ and $\eta$ are the homotopy elements detected by
$\rho$ and $h_1$ respectively.  However, 
neither $2 \rho$ nor $2 \eta$ are zero.}

There are two additional reasons why the element $\h$
plays a central role.  First, it corresponds to the hyperbolic
plane under the isomorphism between
motivic $\pi_{0,0}$ and the Grothendieck--Witt group of symmetric bilinear
forms \cite{Morel-pi0}.  Second, it plays the role of the zeroth 
Hopf map, in the sense that the Steenrod operations on its cofiber
are simpler than the Steenrod operations for the cofiber of $2$.

Consequently, instead of describing motivic stable homotopy groups
as a module over the $2$-adic integers $\Z_2$ (i.e., in terms of
the action of $2$),
it is easier to describe the homotopy groups
in terms of the action of $\h$.

\subsection{Future directions}

Our work points toward several open problems.

\begin{problem}
Compute motivic $v_1$-periodic homotopy over an arbitrary base field.
See Section \ref{subsctn:base-fields} for further discussion.
\end{problem}

\begin{problem}
Recompute the homotopy of $L$ using the $\R$-motivic Adams spectral
sequence.  This would be a useful comparison object for further
computations with the Adams spectral sequence for the $\R$-motivic sphere.
The classical Adams spectral sequence for $j^{\top}$
was studied by Davis \cite{Davis-J}, but it was only
recently computed completely by Bruner and Rognes \cite{bruner-rognes-j}.
We are proposing a motivic analogue of their results.
\end{problem}

\begin{problem}
Carry out the effective spectral sequence for the $\R$-motivic sphere
in a range.  These computations would serve as a useful companion
to $\R$-motivic Adams spectral sequence computations \cite{R-paper}.
The idea is to build on the techniques that are developed
in this manuscript.
\end{problem}

\begin{problem}
Compute the $v_1$-periodic $C_2$-equivariant stable homotopy groups.
More precisely, carry out the $C_2$-effective spectral sequence for a 
$C_2$-equivariant version of $L$.  The details will be similar to
but more
complicated than the computations in this manuscript.  See \cite{Kong20} for the
effective approach to the $C_2$-equivariant version of $\ko$.
Alternatively, one might compute the $v_1$-periodic
$C_2$-equivariant stable homotopy groups by periodicizing the
$v_1$-periodic $\R$-motivic groups with respect to $\tau$,
as considered by Behrens and Shah \cite{BS20}.

Recall that the $\R$-motivic and $C_2$-equivariant stable
homotopy groups are isomorphic in a range \cite{BGI21}.
Consequently, we anticipate that some version of the structure
described in this manuscript appears in the $C_2$-equivariant
context as well.

In the equivariant context, we 
mention
Balderrama's \cite{balderrama} computation
of the homotopy groups of the Borel $C_2$-equivariant
$K(1)$-local sphere, using techniques that are entirely different from ours.
Roughly speaking, Balderrama computes 
the $\tau^4 v_1^4$-periodicization of our result.
\end{problem}

\begin{problem}
Study $K(1)$-localization in the motivic context, which ought
to be something like localization with respect to $KGL/2$.
Compute $K(1)$-local motivic homotopy.
This would provide an intrinsic definition of $v_1$-periodic
homotopy that would improve upon the practical computational
perspective of this manuscript.

A guide to the motivic situation could lie in the work of
Balderrama \cite{balderrama} and Carrick \cite{Carrick}
on equivariant localizations.
\end{problem}

\subsection{Towards $v_1$-periodic homotopy over general base fields}
\label{subsctn:base-fields}

Our explicit computations point the way towards a complete computation
of the $v_1$-periodic motivic stable homotopy groups over arbitrary
fields.
The situation here is analogous to the $\eta$-periodic
$\R$-motivic computations of \cite{GI16}, which foreshadowed the
more general $\eta$-periodic computations of 
\cite{Wilson18}, \cite{OR20}, and \cite{BH20}.

\begin{problem}
\label{prob:j-general}
Let $k$ be an arbitrary field of characteristic different from $2$.
Let $GW(k)$ be the Grothendieck--Witt ring of 
symmetric bilinear forms over $k$.
Describe the 2-primary homotopy groups of the $k$-motivic spectrum $L$ in
terms of the cokernels and kernels of multiplication by
various powers of $2$ and of $\h$ on $GW(k)$.
\end{problem}

Problem \ref{prob:j-general} is stated only in terms of 2-primary
computations because that is the most interesting part.  
We expect that the generalization to odd primes is straightforward.

The exact powers of $2$ and $\h$ 
that are required in Problem \ref{prob:j-general}
depend not only on the coweight but also on the stem. 
Figures \ref{fig:L-Einfty:3mod8} and \ref{fig:L-Einfty:7mod16}
show that
$2^{v(j)+3}$ is the relevant power of $2$ in
most stems in coweight $4j-1$.
Here $v(j)$ is the $2$-adic valuation of $j$, i.e., 
largest number $v$ such that $2^v$ divides $j$.
In coweight $4j-1$ and stem $4i -1$, 
we see larger powers of $2$, as well as powers of $\h$.

Similar observations apply to the kernels that contribute to coweight
$4i$.

\subsection{Outline}

Section \ref{sctn:background} contains some background information
that we will need to get started on our computations.
We briefly discuss convergence of the effective spectral sequences
that we will use.  We recall some results of
Bachmann--Hopkins \cite{BH20} about motivic Adams operations
and of 
Ananyevskiy--R{\"{o}}ndigs--{\O}stv{\ae}r \cite{ARO2020}
about the slices of $\ko$.

The final subsection of Section \ref{sctn:background}  
makes a precise connection between
$\R$-motivic computations and $\C$-motivic computations.
Namely, taking the cofiber of the map $\rho$ changes $\R$-motivic
computations into the analogous $\C$-motivic computations.
This idea originated in \cite{R-paper}.  This  simple
observation has surprisingly powerful consequences,
especially for the analysis of hidden extensions.
It allows us to leverage well-understood $\C$-motivic
information into $\R$-motivic information.

In Section \ref{sctn:background}, we have taken some care to eliminate
details that we do not use.  
In other words, Section \ref{sctn:background} describes the minimal
hypotheses necessary in order to carry out our computations.

Section \ref{sctn:C} considers $\C$-motivic computations, which
play two roles in our work.
First, they serve as a warmup to the
more intricate $\R$-motivic computations.  Second, the comparison
between $\R$-motivic and $\C$-motivic homotopy is a necessary
ingredient for our computations. 
In this section,
we describe the effective spectral sequence for $\ko^\C$.  This material
is well-known, since it is the same (up to regrading) as the
$\C$-motivic Adams--Novikov spectral sequence for $\ko^\C$,
which is nearly the same as the classical Adams--Novikov spectral
sequence for $\ko^\top$.
We then use the fiber sequence
\[
L^\C \tto \ko^\C \ttoo{\psi^3-1} \ko^\C
\]
in order to determine the $E_1$-page of the effective spectral sequence
for $L^\C$.  

We next completely analyze the effective spectral
sequence for the $\eta$-per\-i\-od\-ic\-i\-za\-tion $L^\C[\eta^{-1}]$.
The $\eta$-periodic spectral sequence is significantly simpler than the
unperiodicized spectral sequence.
We note the close similarity between the homotopy of
$L^\C[\eta^{-1}]$ and the computations of Andrew--Miller \cite{AM2017}.

The $\eta$-periodic effective differentials
completely determine the unperiodicized effective differentials for
$L^\C$.  Finally, we determine hidden extensions in the
effective $E_\infty$-page for $L^\C$.

Section \ref{sctn:C} completely computes the homotopy of $L^\C$, but
the effective spectral sequence is not necessarily the simplest
way of obtaining the computation.  Nevertheless, we have chosen this
approach because of its relationship to our later $\R$-motivic 
computations.

Section \ref{sctn:ko-eff} analyzes the effective
spectral sequence for $\R$-motivic $\ko$, including all differentials
and hidden extensions.  The $E_1$-page is readily determined from
the work of
Ananyevskiy--R{\"{o}}ndigs--{\O}stv{\ae}r \cite{ARO2020}
on the slices of $\ko$.
We draw particular attention to the formula
\begin{equation}
\label{eq:ARO}
(\tau h_1)^2 = \tau^2 \cdot h_1^2 + \rho^2 \cdot v_1^2.
\end{equation}
This formula has a major impact on the shape of the answers that
we obtain.  In a sense, our work merely draws algebraic conclusions
from Equation (\ref{eq:ARO}) and $\eta$-periodic information.
The hidden extensions in the effective $E_\infty$-page for $\ko$
are easily determined by comparison to the $\C$-motivic case,
using the relationship between $\C$-motivic and $\R$-motivic homotopy
that is described in Section \ref{subsctn:R-C-compare}.

Our computation of the homotopy of $\R$-motivic $\ko$ is not original.
See \cite{Kong20} for a $C_2$-equivariant analogue of the 
effective spectral sequence for $\ko$. 
The $\R$-motivic computation can be extracted from the $C_2$-equivariant computation by dropping the ``negative cone'' elements. 
Also, 
Hill \cite{Hill-A(1)} computed the Adams spectral sequence for $\ko$,
although the $\R$-motivic spectrum $\ko$ had not yet been constructed at the time.

The next step, undertaken in Section \ref{subsctn:psi3},
is to analyze the effect of $\psi^3$ on the effective spectral
sequence of $\ko$.  This follows from a straightforward comparison
to the classical case, together with careful bookkeeping.
In turn, this leads to a complete understanding of the 
effective $E_1$-page of $L$, which is described in
Section \ref{subsec:slice-L}.  Again, this is mostly a matter of careful
bookkeeping.

Section \ref{subsctn:eta-inverted} completely analyzes the effective
spectral sequence for $\eta$-periodic $L[\eta^{-1}]$.
This information is essentially already well-known, either from
\cite{GI16} or from Ormsby--R\"{o}ndigs \cite{OR20}, although
those references do not specifically mention $L$.

As in the $\C$-motivic situation of Section \ref{sctn:C},
$\eta$-periodic information yields everything that we need to know 
about the unperiodic situation, including all multiplicative
relations in the effective $E_1$-page for $L$ (see Section \ref{subsec:mult-relation}) and
all differentials (see Sections \ref{subsec:d1-diff} and \ref{subsec:higher-diff}).
We again emphasize the significance of Equation (\ref{eq:ARO}) in carrying
out the details.
Finally, Section \ref{subsec:extensions} studies hidden
extensions in the effective $E_\infty$-page for $L$.
As for $\ko$, these hidden extensions follow by comparison
to the $\C$-motivic case.

\subsection{Notation} \label{subsec:notation}

We use the following notation conventions.
\begin{itemize} 
\item
$v(n)$ is the $2$-adic valuation of $n$, i.e., the largest integer
$v$ such that $2^v$ divides $n$.
\item
{Except in Section \ref{sctn:background},}
everything is implicitly $2$-completed.  For example,
$S$ is actually the $2$-complete $\R$-motivic sphere spectrum.
Similarly, $\Z$ is the $2$-adic integers.
\item
$s_*(X)$ are the slices of a motivic spectrum $X$.
\item 
$E_r(X)$ is the $E_r$-page of the effective spectral sequence for a motivic spectrum $X$.
\item
We find the effective slice filtration to be slightly inconvenient for
our purposes.  We prefer to use the ``Adams--Novikov filtration",
which equals twice the effective filtration minus the stem.
\item
Coweight equals the stem minus the motivic weight.
\item 
Elements in $E_r(X)$ are tri-graded. We write $E^{s,f,w}_r(X)$ to denote the part with topological dimension $s$, Adams--Novikov filtration $f$, and motivic weight $w$.
\item
We use unadorned symbols for $\R$-motivic spectra.
For example, $\ko$ is the very effective
cover of the $\R$-motivic Hermitian $K$-theory spectrum.
\item
$X^\C$ is the $\C$-motivic extension-of-scalars spectrum of an
$\R$-motivic spectrum $X$.
\item 
$X^{\top}$ is the Betti realization of an $\R$-motivic spectrum $X$.
\item
$S$ is the $\R$-motivic sphere spectrum.
\item 
$\KO$ is the $\R$-motivic spectrum that represents Hermitian $K$-theory (also known as $\KQ$).
\item
$\ko$ is the very effective connective cover of $\KO$.
\item
$HA$ is the $\R$-motivic Eilenberg--Mac Lane spectrum on the
group $A$.
\item
$\psi^3$ is an Adams operation.  We use the same symbol in the
$\R$-motivic, $\C$-motivic, and classical situations.
\item
$L$ is the fiber of $\ko \ttoo{\psi^3-1} \ko$.
\item
$\Sigma^{s,w} X$ is a (bigraded) suspension of a motivic spectrum $X$.
\item
$\pi_{*,*}(X)$ are the bigraded stable homotopy groups of an
$\R$-motivic or $\C$-motivic spectrum.
\item
Recall that $\epsilon$ is the motivic homotopy class that
is represented by the twist map $S \wedge S \to S \wedge S$,
where $S$ is the motivic sphere spectrum.
Let $\h$ be the element $1-\epsilon$, which corresponds to the hyperbolic
plane under the isomorphism between $\pi_{0,0}(S)$ and the
Grothendieck-Witt ring $GW(\R)$ \cite{Morel-pi0}.
\item
The element $\rho$ belongs to the $\R$-motivic homology of a point.
It is the class represented by
$-1$ in the Milnor $K$-theory of $\R$.
Since $\rho$ survives all of the spectral sequences under consideration,
we use the same symbol for the corresponding homotopy class.
However, there is a choice of homotopy class represented by $\rho$ because
of the presence of elements in higher filtration.
There is an inconsistency in the literature about this choice.
Following \cite{Bach18a}, 
we define $\rho$ such that
$\epsilon = \rho\eta - 1$, or equivalently
$2 = \rho\eta + \h$.
\end{itemize}

We frequently use names for indecomposables that consist of more than
one symbol.  For example, Theorem \ref{thm:ko-slices} discusses the 
indecomposable element $v_1^2$ of the effective $E_1$-page for
$\ko^\C$.  These longer names are slightly more cumbersome.  This is
especially the case when we consider products.  We will use expressions
of the form $x \cdot y$ for clarity.

On the other hand, our names are particularly convenient because they
reflect the origins of the elements in terms of the spectral sequences
that we use.  For example, consider the indecomposable element
$2 v_1^2$ of the effective $E_\infty$-page for $\ko^\C$, as discussed
in Theorem \ref{thm:ko-C-Einfty} (see also Figure \ref{fig:ko-C-Einfty}).
This name reflects the element's origin in the effective $E_1$-page.
It also illuminates relations such as
\[
2v_1^2 \cdot 2v_1^2 = 4 \cdot v_1^4
\]
However, one must be careful about possible error terms in such formulas;
see especially Equation (\ref{eq:ARO}).

\section{Background}
\label{sctn:background}

In this section only, we write $\ko$ for the integral version
of the very effective cover of the Hermitian $K$-theory spectrum,
and we use the usual decorations to indicate localizations and
completions of $\ko$. In the rest of the manuscript,
$\ko$ is assumed to be $2$-completed.

\subsection{The effective slices of $\ko$}
\label{subsec:ko}

We recall the structure of the effective slices 
of $\ko$.

\begin{theorem}[{\cite[Theorem 17]{ARO2020}}] \label{thm:ko-slices}
The slices of $\ko$ are 
$$s_*(\ko) = H\Z[h_1,v_1^2]/(2h_1),$$
where $v_1^2$ and $h_1$ have degrees $(4,0,2)$ and $(1,1,1)$ 
respectively.
\end{theorem}

We explain the expression in Theorem \ref{thm:ko-slices}. Each monomial of degree $(s,f,w)$ contributes a summand of $\Sigma^{s,w}HA$ in the $\left( \frac{s+f}{2} \right)$th slice. Here $HA$ is the motivic Eilenberg--Mac Lane spectrum associated to $A$.  The abelian group $A$ is $\F_2$ when the monomial is $2$-torsion, and is $\Z$ when the monomial is torsion free. We list the first three slices as examples:
\begin{align*}
s_0(\ko) &= H\Z\{1\},
\\s_1(\ko) & =\Sigma^{1,1}H\F_2\{h_1\},
\\s_2(\ko) & =\Sigma^{2,2}H\F_2\{h_1^2\}\vee \Sigma^{4,2}H\Z\{v_1^2\}.
\end{align*}

Beware that the multiplicative structure of $s_*(\ko)$ is not completely
captured by the notation 
in Theorem \ref{thm:ko-slices}.
The essential multiplicative relation 
is Equation (\ref{eq:ARO}), which follows immediately
from the general formulas in \cite{ARO2020}.

\begin{remark}\label{rmk:slice-ANSS}
The calculation of the slices of the motivic sphere spectrum, due to 
R{\"o}ndigs, Spitzweck, and {\O}stv{\ae}r \cite{RSO19}, is commonly expressed at the prime $2$ as
$$ s_*(S) = \MZ\tensor \Ext_{BP_*BP}^{*,*}(BP_*, BP_*). $$
Analogously, Theorem \ref{thm:ko-slices} says that
\[
\label{eq:ko-slices-ANSS} s_*(\ko) = \MZ\tensor \Ext_{BP_*BP}^{*,*}(BP_*,
BP_*(\ko^{\top})).
\]
However, we do not know of
a general theorem relating the slices of a motivic spectrum with the
Adams--Novikov $E_2$-page for its topological counterpart.
\end{remark}

\subsection{The Adams operation $\psi^3$ and the spectrum $L$}

Bachmann and Hopkins \cite{BH20} constructed a motivic analogue 
of the classical Adams operation $\psi^3$.
We summarize the results that we need.

\begin{theorem}[{\cite
{BH20}}]
\label{thm:BH-adams}
There is a unital ring map 
$\psi^3:\ko\left[\frac{1}{3}\right] \to \ko\left[\frac{1}{3}\right]$ whose Betti realization is the classical
Adams operation $\psi^3$.
\end{theorem}

\begin{proof}
There is a unital ring map
$\psi^3: \KO\left[\frac{1}{3}\right] \to \KO\left[\frac{1}{3}\right]$ 
\cite[Theorem 3.1]{BH20}, which is an $E_\infty$-map.
Its Betti realization is also an $E_\infty$-map whose
action on the classical Bott element is multiplication by $81$.
These properties uniquely characterize the classical Adams operation.

Now apply very effective covers, and the result about 
$\ko$ follows formally.
\end{proof}

The original result is more general in more than one sense.
First, it works over
general base schemes in which $2$ is invertible, while we only use
the construction over $\R$.
Second, its values are computed more precisely than just compatibility
with the classical values.

\begin{corollary}
\label{cor:psi3}
\mbox{}
\begin{enumerate} 
\item 
$\psi^3:\pi_{*,*}(\ko^\wedge_2) \to
\pi_{*,*}(\ko^\wedge_2)$ is a ring map.
\item
\label{item:psi(S)} 
If $x$ is in the image of the unit map $\pi_{{*,*}}(S^\wedge_2)\to
\pi_{{*,*}}(\ko^\wedge_2)$, then $\psi^3(x)=x$.
\item
\label{part:psi3-Betti}
There is a commutative diagram
\[
\xymatrix{
\pi_{*,*} (\ko^\wedge_2) \ar[r]^{\psi^3} \ar[d] & 
\pi_{*,*} (\ko^\wedge_2) \ar[d] \\
\pi_* (\ko^\top)^\wedge_2 \ar[r]_{\psi^3} & 
\pi_* (\ko^\top)^\wedge_2,
}
\]
where the vertical maps
are Betti realization homomorphisms.
\end{enumerate}
\end{corollary}

\begin{proof}
These are computational consequences of Theorem \ref{thm:BH-adams}. 
Part (1) follows from the fact that $\psi^3$ is a ring map.
Part (2) follows from the fact that $\psi^3$ is unital.
Part (3) follows from the fact that the Betti realization
of the motivic Adams operation is the classical Adams operation.
\end{proof}

\begin{remark}
Corollary \ref{cor:psi3} can also be stated in a localized sense rather
than completed sense, but we will not need that.
\end{remark}

\begin{definition}
\label{def:j}
Let $L$ be the fiber of the map
$\ko\left[\frac{1}{3}\right]\ttoo{\psi^3-1}\ko\left[\frac{1}{3}\right]$.
\end{definition}

Note that our definition of $L$ is already localized; we do
not consider an integral version.  Except for this section,
$L$ is assumed to be $2$-completed.

The most important point for us is that there is a fiber sequence
\[
L^\wedge_2 \tto \ko^\wedge_2 \ttoo{\psi^3-1} \ko^\wedge_2
\]
of completed spectra since completion preserves fiber sequences.

\subsection{Convergence of the effective spectral sequence}
\label{subsctn:convergence}

The \emph{effective spectral sequence} for a motivic spectrum $X$ denotes the spectral sequence associated to the effective slice filtration of $X$.
We refer to \cite{levine-convergence,RSO19} for details on the construction and properties of this spectral sequence.

The effective slice 
filtration \cite{voevodsky2002open} has truncations $f^q(X)$ and 
quotients (i.e., slices) $s_q(X)$.
The $E_1$-page of the effective spectral sequence is
$\pi_{*,*}(s_*(X))$.  In good cases, it converges to 
the homotopy groups of a completion of $X$.
We also use the very effective slice
filtration \cite{spitzweck2012motivic}, but only to define $\ko$.

The slice functors do not necessarily commute with completions, i.e., $s_*(X)^\wedge_2$ and $s_*(X^\wedge_2)$ are not always equivalent. 
Consequently, we must carefully define the spectral sequences
that we use to study completed spectra.
On the other hand, the effective
slices do interact nicely with localizations
\cite[Corollary 4.6]{Spitzweck2008RelationsBS}.

\begin{theorem}
\label{thm:ko-convergence}
There are strongly convergent spectral sequences
$$ E_1^{s,f,w}(\ko) = \pi_{s,w} \left( s_{\frac{s+f}{2}}(\ko)^\wedge_2 \right)
\implies \pi_{s,w}(\ko^\wedge_2)$$ 
and 
$$ E_1^{s,f,w}(L) = \pi_{s,w} \left( s_{\frac{s+f}{2}}(L)^\wedge_2 \right)
\implies \pi_{s,w}(L^\wedge_2)$$ 
with differentials $d_r:E_r^{s,f,w} \to E_r^{s-1,f+2r-1,w}$.
\end{theorem}
 
We remind the reader that our grading of the effective spectral
sequence is different than the standard grading in the literature.
Briefly, $s$ represents the topological stem, $f$ represents the Adams--Novikov filtration (not the effective filtration), and $w$ represents the
motivic weight.
See Section \ref{subsec:notation} for more discussion.

\begin{proof}
We discuss the spectral sequence for $\ko$ in detail;
most of the argument for $L$ is the same.

Consider the
effective slice tower
\[
f^0(\ko) \leftarrow 
f^1(\ko)\leftarrow 
f^2(\ko)\leftarrow \cdots.
\]
Now take the $2$-completion of this tower to obtain
\[
f^0(\ko)^\wedge_2 \leftarrow f^1(\ko)^\wedge_2 
    \leftarrow f^2(\ko)^\wedge_2 \leftarrow \cdots.
\]
The resulting layers are the same as $s_*(\ko)^\wedge_2$
since completion respects cofiber sequences.
Beware that this is not necessarily the same
as the slice tower of the completion $\ko^\wedge_2$, 
since slices do not interact nicely with completions.
The associated spectral sequence of this tower is the one
described in the statement of the theorem.

It remains to determine the target of the completed spectral sequence.
The limit of the uncompleted slice tower of $\ko$ is equivalent to its 
$\eta$-completion \cite{RSO19}, \cite{ARO2020}, i.e.,
$$\holim f^n(\ko)\simeq \ko^\wedge_\eta.$$ 
Completion respects limits, so
the limit
$\holim (f^n(\ko)^\wedge_2)$ of the completed slice tower is
equivalent to $(\ko^\wedge_\eta)^\wedge_2$, which is equivalent
to $\ko^\wedge_2$ by \cite[Theorem 1]{HKO11}.
Consequently, the completed effective spectral sequence 
of $\ko$ converges to the homotopy of $\ko^\wedge_2$, as desired.

Strong convergence follows from \cite[Theorem 7.1]{Boardman},
which has a technical hypothesis involving derived {$E_\infty$-pages}. 
For $\ko$, this technical hypothesis follows directly from
the computations of Section \ref{sctn:ko-eff}.
For $L$, the technical hypothesis follows directly from
the computations in
Sections \ref{subsec:d1-diff} and \ref{subsec:higher-diff}.
\end{proof}

\begin{remark}
By construction, we have a fiber sequence
\[
s_*(L)^\wedge_2 \tto s_*(\ko)^\wedge_2 \ttoo{\psi^3-1} s_*(\ko)^\wedge_2,
\]
which yields a long exact sequence
\[
\cdots \tto E_1^{s,f,w}(L) \tto E_1^{s,f,w}(\ko) \ttoo{\psi^3-1} E_1^{s,f,w}(\ko) \tto \cdots.
\]
This long exact sequence will be our main tool for computing 
$E_1(L)$ in Section \ref{subsec:slice-L}.
\end{remark}

\subsection{Comparison between $\R$-motivic and $\C$-motivic homotopy}
\label{subsctn:R-C-compare}

Extension of scalars induces a functor 
$(-)^\C$
from $\R$-motivic stable homotopy
theory to $\C$-motivic stable homotopy theory.
Following \cite[Section 3]{R-paper},
we recall its computational interpretation.
These results will play a central role later in Section \ref{subsec:extensions}
when we consider hidden extensions.

The homotopy element $\rho$ maps to zero under extension of scalars.
Therefore, extension of scalars induces a natural homomorphism
\[
\pi_{*,*} (X/\rho) \to \pi_{*,*}( X^\C)
\]
for any $\R$-motivic spectrum $X$.

\begin{proposition}
\label{prop:R-C-compare}
The maps
\[
\pi_{*,*} (\ko/\rho) \to \pi_{*,*}( \ko^\C)
\]
and
\[
\pi_{*,*} (L/\rho) \to \pi_{*,*}( L^\C)
\]
are isomorphisms.
\end{proposition}

\begin{proof}
The argument for the first isomorphism is essentially identical
to the argument in \cite[Section 3]{R-paper}.
We need that the Adams $E_2$-page for $\ko$ is computed by
the cobar complex associated to the subalgebra $A(1)$ of 
the motivic Steenrod algebra, and multiplication by $\rho$
is injective on this cobar complex.

For the second isomorphism, consider the diagram
\[
\xymatrix{
\cdots \ar[r] & \pi_{*,*}(L/\rho) \ar[r] \ar[d] & \pi_{*,*}(\ko/\rho) \ar[r]\ar[d] & \pi_{*,*}(\ko/\rho) \ar[r]\ar[d] & \cdots \\
\cdots \ar[r] & \pi_{*,*}(L^\C) \ar[r] & \pi_{*,*}(\ko^\C) \ar[r] & \pi_{*,*}(\ko^\C) \ar[r] & \cdots.
}
\]
The vertical maps are induced by extension of scalars. The
horizontal maps arise from the fiber sequence
\[
L \tto \ko \ttoo{\psi^3-1} \ko,
\]
so the rows are exact.
We already know that the second and third vertical arrows are isomorphisms,
so the first vertical arrow is also an isomorphism by the Five Lemma.
\end{proof}

\section{$\C$-motivic computations}
\label{sctn:C}

In this section, we carry out a preliminary computation of the
effective spectral sequences for $\ko^\C$ and $L^\C$.
We also consider the $\eta$-periodic spectral sequences.
We are primarily interested in $\R$-motivic computations,
but we will
need to compare our $\R$-motivic computations to their $\C$-motivic
counterparts.

\subsection{The effective spectral sequence for $\ko^\C$}
\label{subsctn:ko-C}

We review the effective spectral sequence for $\ko^\C$.

\begin{proposition}
\label{prop:ko-C-E1}
The effective spectral sequence for $\ko^\C$ takes the form
\[
E_1(\ko^\C) = \Z[\tau, h_1, v_1^2]/ 2 h_1.
\]
\end{proposition}

\begin{proof}
This follows from Theorem \ref{thm:ko-slices} by taking stable
homotopy groups.  There are no possible error terms to complicate
the multiplicative structure.
\end{proof}

Table \ref{tab:ko-C-E1-generators}
lists the generators of $E_1(\ko^\C)$.
Figure \ref{fig:ko-C-E1} depicts $E_1(\ko^\C)$ in graphical form.

\begin{longtable}{lllll}
\caption{Multiplicative generators for $E_1(\ko^\C)$
\label{tab:ko-C-E1-generators} 
} \\
\hline
coweight & $(s, f, w)$ & $x$ & $d_1(x)$ & $\psi^3(x)$ \\
\hline \endfirsthead
\caption[]{Multiplicative generators for $E_1(\ko^\C)$} \\
\hline
coweight & $(s, f, w)$ & $x$ & $d_1(x)$ & $\psi^3(x)$ \\
\hline \endhead
\hline \endfoot
$0$ & $(1,1,1)$ & $h_1$ & & $h_1$ \\
$1$ & $(0,0,-1)$ & $\tau$ & & $\tau$ \\
$2$ & $(4,0,2)$ & $v_1^2$ & $\tau h_1^3$ & $9 v_1^2$
\end{longtable}

\begin{proposition}
\label{prop:ko-C-diff}
Table \ref{tab:ko-C-E1-generators} gives the values of the 
effective $d_1$ differential on the multiplicative generators
of $E_1(\ko^\C)$.
\end{proposition}

\begin{proof}
The 
$\C$-motivic effective spectral sequence is identical to the
$\C$-motivic Ad\-ams--Nov\-i\-kov spectral sequence up to reindexing.
This claim does not appear to be cleanly stated in the literature,
but it is a computational consequence of the weight $0$ result of
\cite[Theorem 1]{Levine15}.
Alternatively, there is only one pattern of effective differentials
that computes the motivic stable homotopy groups of $\ko^\C$, 
which were previously described 
using the $\C$-motivic Adams spectral sequence \cite{IS11}.
\end{proof}

\begin{theorem}
\label{thm:ko-C-Einfty}
The $E_\infty$-page of the 
effective spectral sequence for $\ko^\C$ takes the form
\[
E_\infty(\ko^\C) = \frac{\Z[\tau, h_1, 2 v_1^2, v_1^4]}{2 h_1, 
\tau h_1^3, (2 v_1^2)^2 = 4 \cdot v_1^4}.
\]
\end{theorem}

\begin{proof}
For degree reasons, there can be no higher differentials in the
effective spectral sequence for $\ko^\C$.
\end{proof}

Table \ref{tab:ko-C-Einfty-generators}
lists the multiplicative generators of $E_\infty(\ko^\C)$.
Figure \ref{fig:ko-C-Einfty} depicts $E_\infty(\ko^\C)$ in graphical form.

\begin{longtable}{llll}
\caption{Multiplicative generators for $E_\infty(\ko^\C)$
\label{tab:ko-C-Einfty-generators} 
} \\
\hline
coweight & $(s, f, w)$ & $x$ & $\psi^3(x)$ \\
\hline \endfirsthead
\caption[]{Multiplicative generators for $E_\infty(\ko^\C)$} \\
\hline
coweight & $(s, f, w)$ & $x$ & $\psi^3(x)$  \\
\hline \endhead
\hline \endfoot
$0$ & $(1,1,1)$ & $h_1$ & $h_1$ \\
$1$ & $(0,0,-1)$ & $\tau$ & $\tau$ \\
$2$ & $(4,0,2)$ & $2v_1^2$ & $9 \cdot 2 v_1^2$ \\
$4$ & $(8,0,4)$ & $v_1^4$ & $81 v_1^4$
\end{longtable}

\begin{remark}
\label{rem:ko-C-homotopy}
There are no possible hidden extensions
in $E_\infty(\ko^\C)$ for degree reasons.
Therefore,
Theorem \ref{thm:ko-C-Einfty} describes 
$\pi_{*,*}(\ko^\C)$ as a ring.
\end{remark}

\subsection{The effective $E_1$-page for $L^\C$}

Our next goal is to describe the effective $E_1$-page $E_1(L^\C)$.
First we must study the values of $\psi^3$ on $\ko^\C$.

\begin{lemma}
\label{lem:psi3-Einfty-C-values}
The map $E_\infty(\ko^\C) \to E_\infty(\ko^\C)$ induced by $\psi^3$ on
effective $E_\infty$-pages 
takes the values shown in Table \ref{tab:ko-C-Einfty-generators}.
\end{lemma}

\begin{proof}
All values follow immediately by comparison along Betti 
realization to the values of classical $\psi^3$.
\end{proof}

\begin{lemma}
\label{lem:psi3-E1-C-values}
The map $E_1(\ko^\C) \to E_1(\ko^\C)$ induced by $\psi^3$ on
effective $E_1$-pages 
takes the values shown in Table \ref{tab:ko-C-E1-generators}.
\end{lemma}

\begin{proof}
The values of $\psi^3$ on $E_1(\ko^\C)$ are compatible
with the values of $\psi^3$ on $E_\infty(\ko^\C)$, as 
shown in Table \ref{tab:ko-C-Einfty-generators}
(see also Lemma \ref{lem:psi3-Einfty-C-values}).
This immediately yields all values.
\end{proof}

In order to describe $E_1(L^\C)$, we need some elementary
number theory.

\begin{definition}
Let $v(n)$ be the $2$-adic valuation of $n$, i.e., the 
exponent of the largest
power of $2$ that divides $n$.
\end{definition}

\begin{lemma}
\label{lem:valuation}
\[
v(3^n-1) =
\left\{
\begin{array}{l@{\quad \mathrm{if} \quad}l}
1 & v(n) = 0 \\
2 + v(n) & v(n) > 0
\end{array}
\right.
\]
\end{lemma}

\begin{proof}
Let $n = 2^a \cdot b$, where $b$ is an odd number, so
$v(n) = a$.
Then
\[
3^n - 1 = 
\left(1 + 3^{2^a} + (3^{2^a})^2 + \cdots + (3^{2^a})^{b-1}\right)
(3 - 1) \prod_{i=0}^{a-1} (1 + 3^{2^i}).
\]
The first factor is odd, so it does not contribute to the $2$-adic
valuation.  
The factor $( 1 + 3^{2^i})$ has valuation 1 if $i > 0$,
and it has valuation 2 if $i = 0$.
\end{proof}

\begin{proposition}
\label{prop:L-C-E1}
The chart in Figure \ref{fig:L-C-E1} depicts 
the effective $E_1$-page of $L^\C$.
\end{proposition}

\begin{proof}
The long exact sequence
\[
\cdots \tto E_1(L^\C) \tto E_1(\ko^\C) \ttoo{\psi^3-1} E_1(\ko^\C) \tto \cdots
\]
induces a short exact sequence
\[
0 \tto \Sigma^{-1} C \tto E_1(L^\C) \tto K \tto 0,
\]
where $C$ and $K$ are the cokernel and kernel of 
$E_1(\ko^\C) \ttoo{\psi^3-1} E_1(\ko^\C)$ respectively.
The cokernel and kernel can be computed directly from
the information given in Table \ref{tab:ko-C-E1-generators}
(see also Lemma \ref{lem:psi3-E1-C-values}).

The kernel is additively generated by all multiples of $h_1$ in
$E_1(\ko^\C)$, together with the elements $\tau^k$ for $k \geq 0$.

The cokernel $C$ is nearly the same as $E_1(\ko^\C)$ itself.
We must impose the relations $(3^{2k}-1) v_1^{2k} = 0$ for all $k > 0$.
Lemma \ref{lem:valuation} says that
$3^{2k}-1$ equals $2^{v(2k)+2} \cdot u$, where $u$ is an odd number,
i.e., a unit in our $2$-adic context.
Therefore, the relation
$(3^{2k}-1) v_1^{2k} = 0$ is equivalent to the relation
$2^{v(2k)+2} v_1^{2k} = 0$.
\end{proof}

Table \ref{tab:L-C-E1-generators} lists some elements
of the effective $E_1$-page of $L^\C$.  In fact, these elements
are multiplicative generators for $E_1(L^\C)$.
By inspection, all elements of $E_1(L^\C)$ are of the form
$\tau^a h_1^b x$, for some $x$ in the table.

We use the same notation for elements of $E_1(L^\C)$ and their
images in $E_1(\ko^\C)$.
On the other hand,
we define the elements $\iota x$ of $E_1(L^\C)$ by the property that
they are the image of $x$ under the map 
$\iota: \Sigma^{-1} E_1(\ko) \to E_1(L)$.
For example, the element $1$ of $E_1(\ko)$ maps to $\iota$.

\begin{longtable}{lll}
\caption{Multiplicative generators for $E_1(L^\C)$: $k \geq 0$
\label{tab:L-C-E1-generators} 
} \\
\hline
coweight & $(s, f, w)$ & generator \\
\hline \endfirsthead
\caption[]{Multiplicative generators for $E_1(L^\C)$: $k \geq 0$} \\
\hline
coweight & $(s, f, w)$ & generator  \\
\hline \endhead
\hline \endfoot
$1$ & $(0,0,-1)$ & $\tau$ \\
$2k$ & $(4k+1,1,2k+1)$ & $h_1 v_1^{2k}$ \\
$2k-1$ & $(4k-1,1,2k)$ & $\iota v_1^{2k}$ \\
\end{longtable}

\begin{remark}
Our choice of notation for elements of $E_1(L^\C)$ is helpful for the
particular analysis at hand.  The generators of $E_1(L^\C)$ also have
traditional names from the perspective of the Adams--Novikov spectral
sequence.  Namely, $h_1 v_1^{2k}$ and $\iota v_1^{2k}$ correspond to
$\alpha_{2k+1}$ and $\alpha_{2k/v(8k)}$ respectively.  
However, the $\alpha$-family
perspective is not so helpful for us.
\end{remark}

\subsection{The effective spectral sequence of $L^\C[\eta^{-1}]$}
\label{subsctn:L-C-eta}
Next, we describe the effective spectral sequence of
$L^\C[\eta^{-1}]$.

In the $\eta$-periodic context, the element $h_1$ is a unit.
Therefore, powers of $h_1$ are inconsequential for computational
purposes.  Consequently, we have removed these powers from all
$\eta$-periodic formulas.  The appropriate powers of $h_1$
can be easily reconstructed from the degrees of elements 
(although this reconstruction is typically not necessary).

\begin{proposition}
\label{prop:L-C-eta-E1}
The effective $E_1$-page for $L^\C[\eta^{-1}]$ is given by
\[
E_1(L^\C[\eta^{-1}]) = \F_2[h_1^{\pm 1}, \tau, v_1^2, \iota]/ \iota^2.
\]
\end{proposition}

\begin{proof}
The functors $s_*$ commute with homotopy colimits \cite[Corollary 4.6]{Spitzweck2008RelationsBS}.
Therefore, we can just invert $h_1$ in $E_1(\ko^\C)$ to obtain
\[
E_1(\ko^\C[\eta^{-1}]) = \F_2[h_1^{\pm 1}, \tau, v_1^2].
\]
See Proposition \ref{prop:ko-C-E1} (and Figure \ref{fig:ko-C-E1})
for the description of $E_1(\ko^\C)$.

The map $E_1(\ko^\C[\eta^{-1}]) \ttoo{\psi^3-1} E_1(\ko^\C[\eta^{-1})$
is trivial because $(\psi^3 -1)(h_1) = 0$, as shown in
Table \ref{tab:ko-C-E1-generators} 
(see also Lemma \ref{lem:psi3-E1-C-values}).
Therefore, the long exact sequence
\[
\cdots \tto E_1(L^\C[\eta^{-1}) \tto E_1(\ko^\C[\eta^{-1}]) \ttoo{\psi^3-1} E_1(\ko^\C[\eta^{-1}]) \tto \cdots
\]
implies that
$E_1(L^\C[\eta^{-1}])$ splits as 
\[
E_1(\ko^\C[\eta^{-1}]) \oplus \Sigma^{-1} E_1(\ko^\C[\eta^{-1}]).
\]
This establishes
the additive structure of $E_1(L[\eta^{-1}])$, as well
as most of the multiplicative structure.

The relation $\iota^2 = 0$ is immediate because there are no possible
non-zero values for $\iota^2$.
\end{proof}

\begin{proposition}
\label{prop:L-C-eta-diff}
In the effective spectral sequence for $L^\C[\eta^{-1}]$,
we have $d_1(v_1^2) = \tau$.  The effective differentials
are zero on all other multiplicative generators on all pages.
\end{proposition}

\begin{proof}
The value of $d_1(v_1^2)$ in $E_1(L^\C[\eta^{-1}])$ follows
by comparison of effective spectral sequences along the maps
$L^\C \to L^\C[\eta^{-1}]$ and $L^\C \to \ko^\C$.
Table \ref{tab:ko-C-E1-generators} 
(see also Proposition \ref{prop:ko-C-diff})
gives the value of
$d_1(v_1^2)$ in $E_1(\ko^\C)$ .
\end{proof}

\begin{remark}
\label{rmk:L-C-eta-E1}
The effective spectral sequence for $L^\C[\eta^{-1}]$
is very close to the effective spectral sequence for the
$\eta$-periodic sphere $S^\C[\eta^{-1}]$.
The effective spectral sequence for $S^\C[\eta^{-1}]$ is the same
(up to reindexing) as the motivic Adams--Novikov spectral sequence
for $S^\C[\eta^{-1}]$.
This motivic Adams--Novikov spectral sequence is analyzed in
\cite{AM2017}.  The element $\iota$ is not present in 
$E_1(S^\C[\eta^{-1}])$, but its multiples
$\iota (v_1^2)^k$ are present.
\end{remark}

\subsection{Effective differentials for $L^\C$}

\begin{proposition}
\label{prop:L-C-diff}
Table \ref{tab:L-C-d1} gives the values of the effective $d_1$ differentials
on the multiplicative generators of $E_1(L^\C)$.
There are no higher differentials in the effective spectral
sequence for $L^\C$.
\end{proposition}

\begin{proof}
All of these differentials follow immediately from the 
effective $d_1$ differentials for $L^\C[\eta^{-1}]$,
which are determined by Proposition \ref{prop:L-C-eta-diff}.

For degree reasons, there are no possible higher differentials.
\end{proof}

\begin{longtable}{llll}
\caption{Effective $d_1$ differentials for $L^\C$: $k \geq 0$
\label{tab:L-C-d1}
} \\
\hline
coweight & $(s, f, w)$ & $x$ & $d_1(x)$ \\
\hline \endfirsthead
\caption[]{Effective $d_1$ differentials for $L^\C$: $k \geq 0$} \\
\hline
coweight & $(s, f, w)$ & $x$ & $d_1(x)$ \\
\hline \endhead
\hline \endfoot
$1$ & $(0,0,-1)$ & $\tau$ & \\
$4k$ & $(8k+1, 1, 4k+1)$ & $h_1 v_1^{4k}$ & \\
$4k+2$ & $(8k+5, 1, 4k+3)$ & $h_1 v_1^{4k+2}$ & $\tau h_1^3 \cdot h_1 v_1^{4k}$ \\
$4k-1$ & $(8k-1, 1, 4k)$ & $\iota v_1^{4k}$ & \\
$4k+1$ & $(8k+3, 1, 4k+2)$ & $\iota v_1^{4k+2}$ & $\tau h_1^3 \cdot \iota v_1^{4k}$ 
\end{longtable}

\begin{theorem}
The $E_\infty$-page of the effective spectral sequence for $L^\C$ is
depicted in Figure \ref{fig:L-C-Einfty}.
\end{theorem}

\begin{proof}
Because there are no higher effective differentials for $L^\C$,
we obtain the effective $E_\infty$-page immediately from the
effective $d_1$ differentials in Table \ref{tab:L-C-d1}
(see also Proposition \ref{prop:L-C-diff}).
\end{proof}

\subsection{Hidden extensions in $E_\infty(L^\C)$}

\begin{proposition}
\label{prop:L-C-hidden}
In the effective spectral sequence for $L^\C$, there are hidden
$\h$ extensions from $\iota 4 v_1^{4k+2}$ to 
$\tau h_1^2 \cdot h_1 v_1^{4k}$ for all $k \geq 0$.
\end{proposition}

\begin{proof}
Recall that $\tau \eta^2 = \langle \h, \eta, \h \rangle$
in the homotopy of the $\C$-motivic sphere \cite[Table 7.23]{Isaksen19}.
If $\alpha$ is a homotopy element of $L^\C$ such that $\h \alpha$ is zero,
then
\[
\alpha \cdot \tau \eta^2 = \alpha \langle \h, \eta, \h \rangle =
\langle \alpha, \h, \eta \rangle \h.
\]

In particular, let $\alpha$ be detected by $h_1 v_1^{4k}$.
Then $\tau h_1^2 \cdot h_1 v_1^{4k}$ detects a homotopy element
that is divisible by $\h$, so $\tau h_1^2 \cdot h_1 v_1^{4k}$
must be the target of a hidden $\h$ extension.  There is only 
one possible source for this extension.
\end{proof}

\section{The effective spectral sequence for $\ko$}
\label{sctn:ko-eff}

We now study the effective spectral sequence for $\R$-motivic $\ko$.

\begin{proposition}
\label{prop:ko-E1}
The effective spectral sequence for $\ko$ takes the form
\[
E_1(\ko) = \frac{\Z[\rho, \tau^2, h_1, \tau h_1, v_1^2]}
{2 \rho, 2 h_1, 2 \cdot \tau h_1, 
(\tau h_1)^2 = \tau^2 \cdot h_1^2 + \rho^2 \cdot v_1^2}
\]
\end{proposition}

\begin{proof}
The additive structure
follows from Theorem \ref{thm:ko-slices} by taking stable
homotopy groups.  
We need that the homotopy groups of $\R$-motivic $H\Z$ are
\[
H\Z_{*,*} = \Z[\tau^2, \rho]/ 2 \rho,
\]
and the homotopy groups of $\R$-motivic $H\F_2$ are
\[
(H\F_2)_{*,*} = \F_2[\tau, \rho].
\]
The multiplicative structure is mostly also immediate from
Theorem \ref{thm:ko-slices}.  As explained in
\cite{Kong20}, our formula for $(\tau h_1)^2$
is equivalent to the formula
$\eta^2 \too{\delta} \sqrt{\alpha}$ given in \cite[p.\ 1029]{ARO2020}.
\end{proof}

Table \ref{tab:ko-E1-generators}
lists the generators of $E_1(\ko)$.
Figure \ref{fig:ko-E1} depicts $E_1(\ko)$ in graphical form.

\begin{longtable}{llllll}
\caption{Multiplicative generators for $E_1(\ko)$
\label{tab:ko-E1-generators} 
} \\
\hline
coweight & $(s, f, w)$ & $x$ & $d_1(x)$ & $\psi^3(x)$ & image in $E_1(\ko[\eta^{-1}])$ \\
\hline \endfirsthead
\caption[]{Multiplicative generators for $E_1(\ko)$} \\
\hline
coweight & $(s, f, w)$ & $x$ & $d_1(x)$ & $\psi^3(x)$ & image in $E_1(\ko[\eta^{-1}])$ \\
\hline \endhead
\hline \endfoot
$0$ & $(-1,1,-1)$ & $\rho$ & & $\rho$ & $\rho$ \\
$0$ & $(1,1,1)$ & $h_1$ & & $h_1$ & $1$ \\
$1$ & $(1,1,0)$ & $\tau h_1$ & & $\tau h_1$ & $\tau$ \\
$2$ & $(0,0,-2)$ & $\tau^2$ & $\rho^2 \cdot \tau h_1$ & $\tau^2$ & $\tau^2 + \rho^2 \cdot v_1^2$ \\
$2$ & $(4,0,2)$ & $v_1^2$ & $\tau h_1 \cdot h_1^2$ & $9 v_1^2$ & $v_1^2$
\end{longtable}

\begin{proposition}
\label{prop:ko-diff}
Table \ref{tab:ko-E1-generators} gives the values of the 
effective $d_1$ differential on the multiplicative generators
of $E_1(\ko)$.
\end{proposition}

\begin{proof}
The value of $d_1(\tau^2)$ follows from \cite[Theorem 20]{ARO2020} 
and $\R$-motivic Steenrod algebra actions.  Then
the value of $d_1(v_1^2)$ follows from Equation (\ref{eq:ARO}).

Alternatively, there is only one pattern of effective differentials
that computes the motivic stable homotopy groups of $\ko$,
which were previously computed with the $\R$-motivic Adams
spectral sequence \cite{Hill-A(1)}.
\end{proof}

The entire $d_1$ differential in the effective spectral sequence for $\ko$
can easily be deduced from Proposition \ref{prop:ko-diff}
and the Leibniz rule.

\begin{theorem}
\label{thm:ko-Einfty}
The $E_\infty$-page of the 
effective spectral sequence for $\ko$
is depicted in Figures \ref{fig:ko-Einfty:0}, \ref{fig:ko-Einfty:1},
and \ref{fig:ko-Einfty:2}.
\end{theorem}

\begin{proof}
The Leibniz rule, together with the 
values in Table \ref{tab:ko-E1-generators}
(see also Proposition \ref{prop:ko-diff}), 
completely determines
the effective $d_1$ differential on $E_1(\ko)$.  
The $E_2$-page can then be determined directly.
However, the computation is not entirely straightforward.
Of particular note is the differential
\[
d_1(\tau^2 \cdot \tau h_1 \cdot v_1^2) = 
\tau^4 \cdot h_1^4 + \rho^4 \cdot v_1^4,
\]
which yields the relation
\begin{equation}
\label{eq:t^4h1^4}
\tau^4 \cdot h_1^4 = \rho^4 \cdot v_1^4
\end{equation}
in $E_2(\ko)$.
 
For degree reasons, there can be no higher differentials in the
effective spectral sequence for $\ko$.
\end{proof}

For legibility,
Figures \ref{fig:ko-Einfty:0}, \ref{fig:ko-Einfty:1}, 
and \ref{fig:ko-Einfty:2} display $E_\infty(\ko)$ in three
different charts separated by coweight modulo 4.  There is no chart for
coweights 3 mod 4 because $E_\infty(\ko)$
is zero in those coweights.

Figure \ref{fig:ko-d1:2-1} illustrates part of the
analysis of the
$d_1$ differentials and the determination of $E_2(\ko)$;
it is meant to be representative, not thorough.
The chart shows some of the elements in coweights $1$ and $2$ mod 4,
together with the $d_1$ differentials that relate these elements.
In this chart, one can see that
$\tau^2 \cdot h_1^2 + \rho^2 \cdot v_1^2$ survives to $E_2(\ko)$.
This element survives to $E_\infty(\ko)$.  It is labelled
$(\tau h_1)^2$ in Figure \ref{fig:ko-Einfty:2}, in accordance with
Equation (\ref{eq:ARO}).

\begin{remark}
There is an alternative, slightly more structured, method for
obtaining $E_\infty(\ko)$.  One can filter $E_1(\ko)$ by powers
of $\tau h_1$ and obtain a spectral sequence that converges to
$E_2(\ko)$.  In this spectral sequence, we have the relation
$\tau^2 \cdot h_1^2 = \rho^2 \cdot v_1^2$.  There are differentials
$d_1(\tau^2) = \rho^2 \cdot \tau h_1$ and $d_1(v_1^2) = h_1^2 \cdot \tau h_1$.  Then there is a higher differential $d_3(\tau^2 \cdot v_1^2) = 
(\tau h_1)^3$.  None of this is essential to
our study, but the interested reader may wish to carry out the details.
\end{remark}\hfill

Table \ref{tab:ko-Einfty-generators} lists the multiplicative
generators of $E_\infty(\ko)$.  It is possible to give a complete
list of relations.  However, the long list is not so helpful for
understanding the structure of $E_\infty(\ko)$.  The charts in
Figures \ref{fig:ko-Einfty:0}, \ref{fig:ko-Einfty:1}, and
\ref{fig:ko-Einfty:2} are more useful for this purpose.

\begin{longtable}{llll}
\caption{Multiplicative generators for $E_\infty(\ko)$
\label{tab:ko-Einfty-generators} 
} \\
\hline
coweight & $(s, f, w)$ & $x$ & $\psi^3(x)$  \\
\hline \endfirsthead
\caption[]{Multiplicative generators for $E_\infty(\ko)$} \\
\hline
coweight & $(s, f, w)$ & $x$ & $\psi^3(x)$ \\
\hline \endhead
\hline \endfoot
$0$ & $(-1,1,-1)$ & $\rho$ & $\rho$ \\
$0$ & $(1,1,1)$ & $h_1$ & $h_1$ \\
$1$ & $(1,1,0)$ & $\tau h_1$ & $\tau h_1$ \\
$2$ & $(0,0,-2)$ & $2 \tau^2$ & $2 \tau^2$ \\
$2$ & $(4,0,2)$ & $2 v_1^2$ & $9 \cdot 2 v_1^2$ \\
$4$ & $(0,0,-4)$ & $\tau^4$ & $\tau^4$ \\
$4$ & $(4,0,0)$ & $2 \tau^2 v_1^2$ & $9 \cdot 2 \tau^2 v_1^2$\\
$4$ & $(8,0,4)$ & $v_1^4$ & $81 v_1^4$ \\
\end{longtable}

\begin{proposition}
\label{prop:ko-hidden}
Table \ref{tab:ko-hidden} lists some hidden extensions
by $\rho$, $\h$, and $\eta$ in the effective spectral sequence for
$\ko$.  All other hidden extensions 
by $\rho$, $\h$, and $\eta$ 
are $v_1^{4}$-multiples and $\tau^4$-multiples of these.
\end{proposition}

\begin{proof}
Recall from Proposition \ref{prop:R-C-compare} that the homotopy
of $\ko/\rho$ is isomorphic to the homotopy of $\ko^\C$.
Therefore, we completely understand the homotopy of
$\ko/\rho$ from Theorem \ref{thm:ko-C-Einfty} and
Figure \ref{fig:ko-C-Einfty}.

The hidden $\rho$ extensions follow from
inspection of the long exact sequence associated to the cofiber
sequence
\[
\Sigma^{-1,-1} \ko \ttoo{\rho} \ko \to \ko/\rho.
\]
The map $\ko \to \ko/\rho$ takes
the elements $\tau^4 \cdot h_1^3$ and $(\tau h_1)^2 h_1$ to zero
because there are no possible targets in the homotopy of $\ko/\rho$.
Therefore, those two elements must receive hidden $\rho$ extensions,
and there is only one possibility in both cases.

The relation $\tau^4 \cdot h_1^4 = \rho^4 \cdot v_1^4$
(see Equation (\ref{eq:t^4h1^4}))
 then
implies that $2 \tau^2 v_1^2$ also supports an $h_1$ extension.

The map $\ko/\rho \to \Sigma^{0,-1} \ko$
takes $\tau^3$ and $\tau^3 h_1$ to $2 \tau^2$ and
$\rho (\tau h_1)^2$ respectively.
There is an $h_1$ extension connecting 
$\tau^3$ and $\tau^3 h_1$ in $\ko/\rho$, so there must be a hidden
$\eta$ extension from $2 \tau^2$ to $\rho (\tau h_1)^2$.

The hidden $\h$ extension on $\tau h_1$ follows from the analogous
hidden extension in the homotopy groups of the $\R$-motivic sphere
\cite{DI17}
\cite{R-paper}, using the unit map $S \to \ko$.
Alternatively, this hidden extension is computed
in \cite[Proposition 4.3]{Hill-A(1)} in the context of the
$\R$-motivic Adams spectral sequence for $\ko$.

Finally, multiply by $\tau h_1$ to obtain the hidden
$\h$ extension on $(\tau h_1)^2$.

For degree reasons, there are no other possible hidden extensions to
consider.
\end{proof}

\begin{longtable}{lllll}
\caption{Hidden extensions in $E_\infty(\ko)$
\label{tab:ko-hidden} 
} \\
\hline
coweight & source & type & target & $(s, f, w)$ \\
\hline \endfirsthead
\caption[]{Hidden extensions in $E_\infty(\ko)$} \\
\hline
coweight  & source & type & target & $(s, f, w)$  \\
\hline \endhead
\hline \endfoot
$2$  & $2 v_1^2$ & $\rho$ & $(\tau h_1)^2 h_1$ & $(3 , 3, 1)$ \\
$4$  & $2 \tau^2 v_1^2$ & $\rho$ & $\tau^4 \cdot h_1^3$ & $(3, 3, -1)$ \\
$4$  & $2 \tau^2 v_1^2$ & $\eta$ & $\rho^3 \cdot v_1^4$ & $(5, 3, 1)$ \\
$2$  & $2 \tau^2$ & $\eta$ & $\rho (\tau h_1)^2$ & $(1, 3, -1)$ \\
$1$  & $\tau h_1$ & $\h$ & $\rho \cdot \tau h_1 \cdot h_1$ & $(1, 3, 0)$ \\
$2$  & $(\tau h_1)^2$ & $\h$ & $\rho (\tau h_1)^2 h_1$ & $(2, 4, 0)$ \\
\end{longtable}

\begin{remark}
We have completely analyzed the $E_\infty$-page of the effective spectral
sequence for $\ko$, but this is not quite the same as completely describing
the homotopy of $\ko$.  
In particular, one must choose an element of
$\pi_{*,*} \ko$ that is represented by each
multiplicative generator of $E_\infty(\ko)$ (see Table \ref{tab:ko-Einfty-generators}).
In some cases,
there is more than one choice
because of the presence
of elements in higher filtration in the $E_\infty$-page.
The choices of $\rho$, $h_1$, $\tau h_1$, and $\tau^4$
can be made arbitrarily;
the ring structure is unaffected by these choices.
The elements $2 \tau^2$ and $2 v_1^2$ are already well-defined
because there are no elements in higher filtration.
Finally, the choices of $2 \tau^2 v_1^2$ and $v_1^4$ can then be uniquely
specified by the relations
$\rho \cdot 2 \tau^2 v_1^2 = \tau^4 \cdot h_1^3$ and
$\rho^4 \cdot v_1^4 = \tau^4 \cdot h_1^4$.
\end{remark}

\subsection{$\eta$-periodic $\ko$}
\label{subsctn:ko-eta}

Later we will need some information about 
the $\eta$-periodic spectrum $\ko[\eta^{-1}]$.
As in Section \ref{subsctn:L-C-eta},
powers of $h_1$ are inconsequential for computational
purposes in the $\eta$-periodic context.  
Consequently, we have removed these powers from all
$\eta$-periodic formulas.

\begin{proposition}
\label{prop:ko-eta-E1}
The effective $E_1$-page for $\ko$ is given by
\[
E_1 (\ko[\eta^{-1}]) =
\F_2[h_1^{\pm 1}, \tau, \rho, v_1^2].
\]
Moreover, 
the periodicization map $\ko \to \ko[\eta^{-1}]$ induces the map
on effective $E_1$-pages 
whose values are given in 
Table \ref{tab:ko-E1-generators}.
\end{proposition}

The first part of Proposition \ref{prop:ko-eta-E1} was first proved in
\cite[Theorem 19]{ARO2020}, although the notation is different.

\begin{proof}
The functors $s_*$ commute with homotopy colimits \cite[Corollary 4.6]{Spitzweck2008RelationsBS}.
Therefore, we can just invert $h_1$ in the description
of $E_1(\ko)$ given in Proposition \ref{prop:ko-E1}
(see also Figure \ref{fig:ko-E1}).
\end{proof}

\begin{remark}
\label{rem:t^2}
Table \ref{tab:ko-E1-generators} gives an unexpected value 
for $\tau^2$.
Recall that $\tau^2$ is indecomposable in $E_1(\ko)$, so there is
no inconsistency.
The unexpected value arises from Equation (\ref{eq:ARO}).
\end{remark}

\subsection{The Adams operation $\psi^3$ in effective spectral sequences}
\label{subsctn:psi3}

Our goal in this section is to study $\psi^3$ as a map
of effective spectral sequences.  This will allow us to compute
the $E_1$-page of the effective spectral sequence for $L$.

\begin{lemma}
\label{lem:psi3-Einfty-values}
The map $E_\infty(\ko) \to E_\infty(\ko)$ induced by $\psi^3$ on
effective $E_\infty$-pages 
takes the values shown in Table \ref{tab:ko-Einfty-generators}.
\end{lemma}

\begin{proof}
Corollary \ref{cor:psi3}(2) gives the values of $\psi^3$ on
$\rho$, $h_1$, and $\tau h_1$.

The value of $\psi^3$ on $\tau^4$ is determined
immediately by comparison along Betti realization to the
classical value $\psi^3(1) = 1$.
The computation is greatly simplified
by ignoring terms in higher effective filtration.
Similarly, the value of $\psi^3$ on $2 \tau^2$ is determined
by the classical value $\psi^3(2) = 2$.

The remaining values in Table \ref{tab:ko-Einfty-generators}
are also determined by comparison along Betti realization to 
the classical values $\psi^3(2 v_1^2) = 9 \cdot 2 v_1^2$ and
$\psi^3(v_1^4) = 81 v_1^4$.
\end{proof}

\begin{lemma}
\label{lem:psi3-E1-values}
The map $E_1(\ko) \to E_1(\ko)$ induced by $\psi^3$ on
effective $E_1$-pages 
takes the values shown in Table \ref{tab:ko-E1-generators}.
\end{lemma}

\begin{proof}
The values of $\psi^3$ on $E_1(\ko)$ are compatible
with the values of $\psi^3$ on $E_\infty(\ko)$, as 
shown in Table \ref{tab:ko-Einfty-generators}.
This immediately yields the value of $\psi^3$ on $\rho$, $h_1$, and 
$\tau h_1$.

The value of $\psi^3((\tau^2)^2)$ must be $(\tau^2)^2$
by compatibility
with the value of $\psi^3(\tau^4)$ in $E_\infty(\ko)$.
Then the relation
$\psi^3((\tau^2)^2) = (\psi^3(\tau^2))^2$ implies that
$\psi^3(\tau^2) = \tau^2$.

Similarly,
the value of $\psi^3((v_1^2)^2)$ must be $81 (v_1^2)^2$ by
compatibility with the value of $\psi^3(v_1^4)$ in $E_\infty(\ko)$.
Then the relation 
$\psi^3((v_1^2)^2) = (\psi^3(v_1^2))^2$ implies that
$\psi^3(v_1^2) = 9 v_1^2$.
\end{proof}

\begin{remark}
\label{rmk:psi3-E1-values}
Since $\psi^3$ is a ring homomorphism, all values of $\psi^3$
on $E_1(\ko)$ 
are readily determined by the values on multiplicative generators
given in Table \ref{tab:ko-E1-generators}.
In particular, for all $k \geq 0$,
\[
\psi^3(v_1^{2k}) = 9^k v_1^{2k}.
\]
\end{remark}

\begin{remark}
{
Table \ref{tab:ko-E1-generators} implies that
$\psi^3(v_1^4) = 81 v_1^4$.
The careful reader will notice that this expression
appears to be simpler than
the analogous formula in \cite[Theorem 3.1(2)]{BH20}.}
The difference is explained by the fact that
we are working only up to higher filtration.  In particular,
our formulas do not reflect the difference between $2$ and $\h$.
This also means that our formulas are less precise, but that has
no consequence for our computational results.
\end{remark}

\section{The effective spectral sequence for $L$}

\subsection{The effective $E_1$-page of $L$}
\label{subsec:slice-L}

In this section we
compute the $E_1$-page of the effective spectral sequence for $L$.

The fiber sequence $L \to \ko\ttoo{\psi^3-1} \ko$ induces a fiber sequence 
\[
s_* L \tto s_* \ko \ttoo{\psi^3-1} s_* \ko
\]
on slices.  Upon taking homotopy groups, we obtain a long exact 
sequence
\[
\cdots \tto E_1(L) \tto E_1(\ko) \ttoo{\psi^3-1} E_1(\ko) \tto \cdots.
\]
Table \ref{tab:ko-E1-generators} (see also Lemma \ref{lem:psi3-E1-values}) gives us complete computational
knowledge of the map
$E_1(\ko) \to E_1(\ko)$.
This allows us to compute $E_1(L)$.

\begin{proposition}
\label{prop:L-E1}
The chart in Figure \ref{fig:L-E1} depicts the effective $E_1$-page of $L$.
\end{proposition}

\begin{proof}
The long exact sequence
\[
\cdots \tto E_1(L) \tto E_1(\ko) \ttoo{\psi^3-1} E_1(\ko) \tto \cdots
\]
induces a short exact sequence
\[
0 \tto \Sigma^{-1} C \tto E_1(L) \tto K \tto 0,
\]
where $C$ and $K$ are the cokernel and kernel of 
$E_1(\ko) \ttoo{\psi^3-1} E_1(\ko)$.
The cokernel and kernel can be computed directly from
the information given in Lemma \ref{lem:psi3-E1-values}.
See also Remark \ref{rmk:psi3-E1-values}.

The kernel consists of all elements in $E_1(\ko)$ with the
exception of the integer multiples of $\tau^{2j} \cdot v_1^{2k}$ 
for $j \geq 0$ and $k > 0$.

The cokernel $C$ is nearly the same as $E_1(\ko)$ itself.
We must impose the relations $(3^{2k}-1) v_1^{2k} = 0$ for all $k > 0$.
Lemma \ref{lem:valuation} says that
$3^{2k}-1$ equals $2^{v(2k)+2} \cdot u$, where $u$ is an odd number,
i.e., a unit in our $2$-adic context.
Therefore, the relation
$(3^{2k}-1) v_1^{2k} = 0$ is equivalent to the relation
$2^{v(2k)+2} v_1^{2k} = 0$.
\end{proof}

Table \ref{tab:L-E1-generators} lists some elements
of the effective $E_1$-page of $L$.  In fact, by inspection
these elements
are multiplicative generators for $E_1(L)$.

We use the same notation for elements of $E_1(L)$ and their
images in $E_1(\ko)$.
On the other hand,
we define the element $\iota x$ of $E_1(L)$ to be the
image of $x$ under the map 
$\iota: \Sigma^{-1} E_1(\ko) \to E_1(L)$.
For example, the element $1$ of $E_1(\ko)$ maps to $\iota$ in $E_1(L)$.

\begin{longtable}{llll}
\caption{Multiplicative generators for $E_1(L)$: $k \geq 0$
\label{tab:L-E1-generators} 
} \\
\hline
coweight & $(s, f, w)$ & generator & image in $E_1(L[\eta^{-1}])$ \\
\hline \endfirsthead
\caption[]{Multiplicative generators for $E_1(L)$: $k \geq 0$} \\
\hline
coweight & $(s, f, w)$ & generator & image in $E_1(L[\eta^{-1}])$ \\
\hline \endhead
\hline \endfoot
$2$ & $(0,0,-2)$ & $\tau^2$ & $\tau^2 + \rho^2 \cdot v_1^2$ \\
$2k+1$ & $(4k+1,1,2k)$ & $\tau h_1 v_1^{2k}$ & $\tau (v_1^2)^k$ \\
$2k$ & $(4k-1,1,2k-1)$ & $\rho v_1^{2k}$ & $\rho (v_1^2)^k$ \\
$2k$ & $(4k+1,1,2k+1)$ & $h_1 v_1^{2k}$ & $(v_1^2)^k$ \\
$2k-1$ & $(4k-1,1,2k)$ & $\iota v_1^{2k}$ & $\iota (v_1^2)^k$ \\
\end{longtable}

\subsection{The effective spectral sequence for $L[\eta^{-1}]$}
\label{subsctn:eta-inverted}

In Section \ref{subsec:slice-L}, we determined the
effective $E_1$-page of $L$.  The next steps in the analysis
of the effective spectral sequence for $L$ are to determine
the multiplicative structure of $E_1(L)$ 
(see Section \ref{subsec:mult-relation})
and to determine the effective differentials 
(see Sections \ref{subsec:d1-diff} and \ref{subsec:higher-diff}).

Before doing so, we collect some information
on the $\eta$-periodicization $L[\eta^{-1}]$.
We will study $L[\eta^{-1}]$ by comparing to the more easily
understood $\ko[\eta^{-1}]$.

As in Sections \ref{subsctn:L-C-eta} and \ref{subsctn:ko-eta},
powers of $h_1$ are inconsequential for computational
purposes in the $\eta$-periodic context.  
Consequently, we have removed these powers from all
$\eta$-periodic formulas.

\begin{proposition}
\label{prop:L-eta-E1}
The effective $E_1$-page for $L[\eta^{-1}]$ is given by
\[
E_1 (L[\eta^{-1}]) =
\F_2[\tau, \rho, v_1^2, \iota]/\iota^2.
\]
Moreover, 
the periodicization map $L \to L[\eta^{-1}]$ induces the map
$E_1(L) \to E_1(L[\eta^{-1}])$ whose values are given in 
Table \ref{tab:L-E1-generators}.
\end{proposition}

\begin{proof}
As in Proposition \ref{prop:ko-eta-E1}, we can just invert $h_1$
in the additive description of $E_1(L)$ given in 
Proposition \ref{prop:L-E1}.

The map $E_1(\ko[\eta^{-1}]) \ttoo{\psi^3-1} E_1(\ko[\eta^{-1})$
is trivial because 
$(\psi^3-1)(h_1) = 0$, as shown in
Table \ref{tab:ko-E1-generators} (see also Lemma \ref{lem:psi3-E1-values}).
Therefore, the long exact sequence
\[
\cdots \tto E_1(L[\eta^{-1}) \tto E_1(\ko[\eta^{-1}]) \ttoo{\psi^3-1} E_1(\ko[\eta^{-1}]) \tto \cdots
\]
splits as
\[
E_1(L[\eta^{-1}]) \cong
E_1(\ko[\eta^{-1}]) \oplus \Sigma^{-1} E_1(\ko[\eta^{-1}]).
\]
With Proposition \ref{prop:ko-eta-E1}, this establishes
the additive structure of $E_1(L[\eta^{-1}])$, as well
as most of the multiplicative structure.

The relation $\iota^2 = 0$ is immediate because there are no possible
non-zero values for $\iota^2$.
\end{proof}

\begin{remark}
As in Remark \ref{rem:t^2},
Table \ref{tab:L-E1-generators} gives an unexpected value for $\tau^2$,
which arises from Equation (\ref{eq:ARO}).
Also,
the last column of Table \ref{tab:L-E1-generators} leaves out
of $h_1$ for readability.
\end{remark}

\begin{remark}
Note that $E_1(L[\eta^{-1}])$ is very close to the
effective $E_1$-page for the $\eta$-periodic sphere $S[\eta^{-1}]$
\cite[Theorem 2.32]{RSO19} \cite[Theorem 2.3]{OR20}.
The element $\iota$ is not present in $E_1(S[\eta^{-1}])$, but 
the elements $\iota v_1^{2k}$ are present.
\end{remark}

\begin{proposition}
\label{prop:L-eta-diff}
Some values of the differentials in the effective spectral sequence of $L[\eta^{-1}]$ are:
\begin{enumerate}
\item
$d_1(v_1^2) = \tau$.
\item
$d_{n+1}(v_1^{2^n}) = \rho^{n+1} \cdot \iota v_1^{2^n}$ for $n \geq 2$.
\end{enumerate}
The effective differentials are zero on all other multiplicative generators
on all pages.
\end{proposition}

Following our convention throughout this section, we have omitted
the powers of $h_1$ from the formulas in Proposition \ref{prop:L-eta-diff}.
	
\begin{proof}
The $d_1$ differential follows from \cite[Lemma 4.2]{RSO19}
or \cite[Theorem 2.6]{OR20}.

To study the higher differentials, consider
the map $S[\eta^{-1}] \to L[\eta^{-1}]$. 
This map induces an isomorphism on stable homotopy groups,
except in coweight $-1$.  This follows from a minor adjustment
to \cite[Theorem 1.1]{BH20}.  The adjustment arises 
from the fact that our $L[\eta^{-1}]$ is the fiber 
of $\ko[\eta^{-1}] \ttoo{\psi^3-1} \ko[\eta^{-1}]$,
while \cite[Theorem 1.1]{BH20} refers to the fiber
of $\ko[\eta^{-1}] \ttoo{\psi^3-1} \Sigma^{8,4} \ko[\eta^{-1}]$.

The homotopy of $S[\eta^{-1}]$
is completely computed in \cite{GI16}, so the homotopy
of $L[\eta^{-1}]$ is known (except in coweight $-1$).
There is only one pattern of differentials that is compatible with the
known values for $L[\eta^{-1}]$.
\end{proof}

\begin{remark}
\label{rmk:profile}
In the language of \cite[Section 4]{OR20}, 
Proposition \ref{prop:L-eta-diff} establishes the profile 
of the $\eta$-periodic effective spectral sequence over $\R$.
\end{remark}

\subsection{Multiplicative relations for $E_1(L)$}
\label{subsec:mult-relation}

In this section, we will completely describe the
product structure on $E_1(L)$.  We do not need all of this 
structure for our later computations, but we include it
for completeness.

\begin{proposition}
\label{prop:E1-products}
Table \ref{tab:E1-products} lists some products in $E_1(L)$.
\end{proposition}

\begin{longtable}{l|llll}
\caption{Products in $E_1(L)$: $j \geq 0$ and $k \geq 0$
\label{tab:E1-products} 
} \\
& $\rho v_1^{2j}$ & $h_1 v_1^{2j}$ & $\tau h_1 v_1^{2j}$ & $\iota v_1^{2j}$ \\
\hline \endfirsthead
\caption[]{Products in $E_1(L)$: $j \geq 0$ and $k \geq 0$} \\
& $\rho v_1^{2j}$ & $h_1 v_1^{2j}$ & $\tau h_1 v_1^{2j}$ & $\iota v_1^{2j}$ \\
\hline \endhead
\hline \endfoot
$\rho v_1^{2k}$ & $\rho \cdot \rho v_1^{2j+2k}$ \\
$h_1 v_1^{2k}$ & $\rho \cdot h_1 v_1^{2j+2k}$ & $h_1 \cdot h_1 v_1^{2j+2k}$ \\
$\tau h_1 v_1^{2k}$ & $\rho \cdot \tau h_1 v_1^{2j+2k}$ & $h_1 \cdot \tau h_1 v_1^{2j+2k}$ & $\tau^2 \cdot h_1 \cdot h_1 v_1^{2j+2k} + $ \\
& & & $+ \rho \cdot \rho v_1^{2j+2k+2}$ \\
$\iota v_1^{2k}$ & $\rho \cdot \iota v_1^{2j+2k}$ & $h_1 \cdot \iota v_1^{2j+2k}$ & $\tau h_1 \cdot \iota v_1^{2j+2k}$ & $0$ \\
\end{longtable}

\begin{proof}
All of these products are detected by
$E_1(L[\eta^{-1}])$, which is described in
Proposition \ref{prop:L-eta-E1}.
We need the values of the periodicization map
$E_1(L) \to E_1(L[\eta^{-1}])$ given in
Table \ref{tab:L-E1-generators}.
\end{proof}

\subsection{The effective $d_1$ differential for $L$}
\label{subsec:d1-diff}

Our next task is to compute the
differentials in the effective spectral sequence for $L$.

\begin{proposition}
\label{prop:d1(j)}
Table \ref{tab:L-d1} gives the values of the
effective $d_1$ differential on the multiplicative generators
of $E_1(L)$.  
\end{proposition}

\begin{longtable}{llll}
\caption{Effective $d_1$ differentials for $L$: $k \geq 0$
\label{tab:L-d1}
} \\
\hline
coweight & $(s, f, w)$ & $x$ & $d_1(x)$ \\
\hline \endfirsthead
\caption[]{Effective $d_1$ differentials for $L$: $k \geq 1$} \\
\hline
coweight & $(s, f, w)$ & $x$ & $d_1(x)$ \\
\hline \endhead
\hline \endfoot
$2$ & $(0,0,-2)$ & $\tau^2$ & $\rho^2 \cdot \tau h_1$ \\
$4k$ & $(8k-1, 1, 4k-1)$ & $\rho v_1^{4k}$ \\
$4k+2$ & $(8k+3, 1, 4k+1)$ & $\rho v_1^{4k+2}$ & $\rho h_1^2 \cdot \tau h_1 v_1^{4k}$ \\
$4k$ & $(8k+1, 1, 4k+1)$ & $h_1 v_1^{4k}$ & \\
$4k+2$ & $(8k+5, 1, 4k+3)$ & $h_1 v_1^{4k+2}$ & $h_1^3 \cdot \tau h_1 v_1^{4k}$ \\
$4k+3$ & $(8k+5, 1, 4k+2)$ & $\tau h_1 v_1^{4k+2}$ & $\tau^2 \cdot h_1^3 \cdot h_1 v_1^{4k} + \rho^2 h_1 \cdot h_1 v_1^{4k+2}$ \\
$4k+1$ & $(8k+1, 1, 4k)$ & $\tau h_1 v_1^{4k}$ & \\
$4k+1$ & $(8k+3, 1, 4k+2)$ & $\iota v_1^{4k+2}$ & $\tau h_1 \cdot h_1^2 \cdot \iota v_1^{4k}$ \\
$4k-1$ & $(8k-1, 1, 4k)$ & $\iota v_1^{4k}$ & \\
\end{longtable}

\begin{proof}
All of these differentials follow immediately from
the effective $d_1$ differentials for $L[\eta^{-1}]$, 
which are all determined by
Proposition \ref{prop:L-eta-diff}(1)
Beware that the exact values of the map
$E_1(L) \to E_1(L[\eta^{-1}])$, 
as shown in Table \ref{tab:L-E1-generators}, are important.

For example, consider the differential
on the element $\tau h_1 v_1^{4k+2}$.  It maps to
$\tau (v_1^2)^{2k+1}$ in $E_1(L[\eta^{-1}])$ (up to $h_1$ multiples,
which as usual we ignore in the $\eta$-periodic situation).
The $\eta$-periodic differential on this latter element
is $\tau^2 (v_1^2)^{2k}$.
Finally, we need to find an element of $E_1(L)$ in the correct degree
whose $\eta$-per\-i\-od\-ic\-iz\-a\-tion is 
$\tau^2 (v_1^2)^{2k}$,
The only possibility
is $\tau^2 \cdot h_1^3 \cdot h_1 v_1^{4k} + 
\rho^2 h_1 \cdot h_1 v_1^{4k+2}$.
\end{proof}

\begin{remark}
All $d_1$ differentials in $E_1(L)$ can be deduced from the information
in Table \ref{tab:L-d1} and the Leibniz rule, but the computations
can be complicated by the multiplicative relations 
of Table \ref{tab:E1-products}.
For example,
\[
d_1(\tau^2 \cdot \tau h_1 v_1^2) = 
\rho^2 \cdot \tau h_1 \cdot \tau h_1 v_1^2 +
\tau^2 (\tau^2 \cdot h_1^4 + \rho^2 h_1 \cdot h_1 v_1^2) =
\tau^4 \cdot h_1^4 + \rho^4 \cdot v_1^4.
\]
\end{remark}

Having completely analyzed the slice $d_1$ differentials
for $E_1(L)$, it is now possible to compute the 
$E_2$-page of the slice spectral sequence for $L$.

\begin{proposition}
\label{prop:L-E2}
The $E_2$-page of the effective spectral sequence for $L$ is depicted
in Figures \ref{fig:L-E2:0}, \ref{fig:L-E2:3}, \ref{fig:L-Einfty:1},
and \ref{fig:L-Einfty:2}.
\end{proposition}

For legibility, Figures 
\ref{fig:L-E2:0}, \ref{fig:L-E2:3}, \ref{fig:L-Einfty:1}, and \ref{fig:L-Einfty:2} display $E_2(L)$ in four different charts
separated by coweight modulo $4$.  
Note that Figures \ref{fig:L-Einfty:1} and \ref{fig:L-Einfty:2}
also serve as $E_\infty$-page charts in coweights
$1$ and $2$ modulo $4$ because there are no higher differentials
that affect these coweights.

\begin{proof}
The Leibniz rule, together with the values in Table \ref{tab:L-d1},
completely determines the effective $d_1$ differential on $E_1(L)$.
The $E_2$-page can then be determined directly.  However,
as in the proof of Theorem \ref{thm:ko-Einfty},
the computation is not entirely straightforward.

It turns out that the $d_1$ differential preserves the image of the map
$\Sigma^{-1} E_1(\ko) \to E_1(L)$.  
Moreover, it turns out that all $d_1$ differentials with values
in the image of 
$\Sigma^{-1} E_1(\ko) \to E_1(L)$ also have source in this image.
(This is not for formal reasons; in fact, the higher effective
differentials do not have this property.)
Consequently, the determination of the $E_2$-page splits into two
separate computations: one for the image of
$\Sigma^{-1} E_1(\ko) \to E_1(L)$, and one for the cokernel
of the same map.

In more concrete terms, we can determine $E_2(L)$ by first
considering only elements of the form $\iota x$, and then separately
considering only elements that are not of this form.

The $d_1$ differential on the image of 
$\Sigma^{-1} E_1(\ko) \to E_1(L)$ is
identical to the $d_1$ differential for $\ko$ 
discussed in Section \ref{sctn:ko-eff}.
The $d_1$ differential on the cokernel of
$\Sigma^{-1} E_1(\ko) \to E_1(L)$ 
is similar to the $d_1$ differential
on $E_1(\ko)$, but slightly different.
The difference is created by the absence of the elements
$v_1^{2k}$ in $E_1(L)$.
\end{proof}

\subsection{Higher differentials}
\label{subsec:higher-diff}
We now consider the higher differentials
in the effective spectral sequence for $L$.

By inspection of the charts for $E_2(L)$,
the only possible higher differentials have source in coweight
congruent to $0$ modulo $4$ and value in coweight congruent to
$3$ modulo $4$.
In other words, in coweights congruent to $1$ and $2$ modulo 4,
we have that $E_2(L)$ equals $E_\infty(L)$.

It turns out that there are many higher differentials.
In fact, nearly all
of the elements in $E_2(L)$ in coweight congruent to 0 modulo 4 support
differentials.  While it is possible to write down explicit formulas
for all of these differentials, the formulas would be cumbersome
and not so helpful.  Rather, we give a more qualitative description
of the differentials because it is more useful for computation.

\begin{proposition}
\label{prop:L-higher-diff}
Consider the elements of $E_2(L)$ in coweights
congruent to $0$ modulo $4$ that belong to the cokernel
of the map $\Sigma^{-1} E_2(\ko) \to E_2(L)$.
\begin{enumerate}
\item
The only permanent cycles are the multiples of $1$, 
the multiples of $2 \tau^{4k}$ for $k \geq 0$, and
$\rho^a h_1^b$ for all $a \geq 0$ and $b \geq 0$.
\item
Excluding the elements listed in (1), 
if an element has coweight congruent to $2^{r-1}$ modulo $2^r$,
then it supports a $d_r$ differential.
\end{enumerate}
\end{proposition}

Proposition \ref{prop:L-higher-diff} may seem imprecise because it
does not give the values of the differentials.  However, there is only
one non-zero possible value in every case, so there is no ambiguity.

\begin{proof}
These differentials follow immediately from the
$\eta$-periodic differentials of Proposition \ref{prop:L-eta-diff},
together with multiplicative relations in $E_2(L)$.

For example,
consider the element $\tau^8 \cdot \rho v_1^{12}$ in coweight $20$,
which is congruent to $2^2$ modulo $2^3$.
Using Table \ref{tab:L-E1-generators}, we find that this element
maps to $\rho^9 (v_1^2)^{10}$ in $E_2(L[\eta^{-1}])$.
Here we are using that $\tau^2$ is zero in 
$E_2(L[\eta^{-1}])$ since it is hit by an $\eta$-periodic
$d_1$ differential.
Proposition \ref{prop:L-eta-diff} says that this element
supports an $\eta$-periodic $d_3$ differential.  It follows
that $\tau^8 \cdot \rho v_1^{12}$ also supports a $d_3$ differential.
\end{proof}

\begin{theorem}
\label{thm:L-Einfty}
The $E_\infty$-page of the effective spectral sequence for
$L$ is depicted in Figures \ref{fig:L-Einfty:0},
\ref{fig:L-Einfty:1}, \ref{fig:L-Einfty:2}, \ref{fig:L-Einfty:-1},
\ref{fig:L-Einfty:3mod8}, \ref{fig:L-Einfty:7mod16}, 
and \ref{fig:L-Einfty:15mod32}.
\end{theorem}

\begin{proof}
The $E_\infty$-page can be deduced directly from the
higher differentials described in
Proposition \ref{prop:L-higher-diff}.
\end{proof}

The $E_\infty$-page in coweights congruent to 
$3$ modulo $4$ is by far the most complicated case.
Figures \ref{fig:L-Einfty:3mod8}, \ref{fig:L-Einfty:7mod16}, 
and \ref{fig:L-Einfty:15mod32} display $E_\infty(L)$
in coweights congruent to 3 modulo 8,
7 modulo 16, and 15 modulo 32 respectively.

In each case (and more generally in coweights congruent to
$2^{n-1}-1$ modulo $2^n$, we see similar patterns with minor variations.
The lower boundary of each chart takes the same shape.
The upper boundary of the $\tau$-periodic portion of each chart
also takes the same shape.  However, the filtration jump between
the lower and upper boundaries increases linearly with $n$.

In addition to the $\tau$-periodic portion of each chart, 
there are also $\tau$-torsion, $\eta$-periodic regions.
These consist of bands of infinite $h_1$-towers of width $n$
that repeat every $2^{n+1}$ stems.
The first such band starts at $\iota v_1^{2^{n-1}}$.

\subsection{Hidden extensions}
\label{subsec:extensions}

Our last goal is to compute hidden extensions by
$\rho$, $\h$, and $\eta$.  See \cite[Section 4.1]{Isaksen19}
for a precise definition of a hidden extension.  
Fortunately, none of the complications
associated with crossing extensions occur in this manuscript.

\begin{proposition}
Table \ref{tab:L-hidden} lists some hidden extensions
by $\rho$, $\h$, and $\eta$ in the effective spectral sequence for $L$.
\end{proposition}

\begin{proof}
The last column of Table \ref{tab:L-hidden} indicates the reason
for each hidden extension.
Some of the hidden extensions follow from the analogous extensions
for $\ko$ given in Table \ref{tab:ko-hidden}, using the maps
$\Sigma^{-1} \ko \to L$ and $L \to \ko$.

Other extensions follow from the long exact sequence associated
to the cofiber sequence
\[
\Sigma^{-1,-1} L \ttoo{\rho} L \ttoo{} L/\rho.
\]
Here we need that the homotopy of $L/\rho$ is isomorphic to the
homotopy of $L^\C$, as shown in Proposition \ref{prop:R-C-compare}.
For example, the hidden $\h$ extensions of 
Proposition \ref{prop:L-C-hidden} give hidden $\h$ extensions
in $L/\rho$, which then imply the hidden
extension from $\iota 4 v_1^2$ to $h_1^2 \cdot \tau h_1$.

\end{proof}

\begin{longtable}{llllll}
\caption{Hidden extensions in $E_\infty(L)$
\label{tab:L-hidden} 
} \\
\hline
coweight & source & type & target & $(s, f, w)$ & proof  \\
\hline \endfirsthead
\caption[]{Hidden extensions in $E_\infty(L)$} \\
\hline
coweight & source & type & target & $(s, f, w)$ & proof \\
\hline \endhead
\hline \endfoot
$0$  & $\iota \cdot \tau h_1$ & $\h$ & $\iota \cdot \rho h_1 \cdot \tau h_1$ & $(0, 2, 0)$ & $\Sigma^{-1} \ko \to L$  \\ 
$1$  & $\tau h_1$ & $\h$ & $\rho h_1 \cdot \tau h_1$ & $(1, 1, 0)$ & $L \to \ko$ \\ 
$1$  & $\iota (\tau h_1)^2$ & $\h$ & $\iota \cdot \rho h_1 (\tau h_1)^2$ & $(1, 3, 0)$ & $\Sigma^{-1} \ko \to L$ \\
$1$  & $\iota \cdot 2 \tau^2$ & $\eta$ & $\iota \cdot \rho (\tau h_1)^2$ & $(-1, 1, -2)$ & $\Sigma^{-1} \ko \to L$ \\
$1$  & $\iota 2 v_1^2$ & $\rho$ & $\iota \cdot h_1 (\tau h_1)^2$ & $(3, 1, 2)$ & $\Sigma^{-1} \ko \to L$ \\
$1$ & $\iota 4 v_1^2$ & $\h$ & $h_1^2 \cdot \tau h_1$ & $(3, 1, 2)$ & $L/\rho$ \\
$2$ & $(\tau h_1)^2$ & $\h$ & $\rho h_1 (\tau h_1)^2$ & $(2, 2, 0)$ & $L \to \ko$ \\
$3$ & $\iota 4 \tau^2 v_1^2$ & $\h$ & $(\tau h_1)^3$ & $(3, 1, 0)$ & $L/\rho$ \\
$3$ & $\iota 2 \tau^2 v_1^2$ & $\rho$ & $\iota \tau^4 \cdot h_1^3$ & $(3, 1, 0)$ & $\Sigma^{-1} \ko \to L$ \\
$3$ & $\iota 2 \tau^2 v_1^2$ & $\eta$ & $\rho^3 \cdot \iota v_1^4$ & $(3, 1, 0)$ & $\Sigma^{-1} \ko \to L$ \\
\hline
$2$ & $2 \tau^2$ & $\eta$ & $\rho (\tau h_1)^2$ & $(0, 0, -2)$ & $L \to \ko$ \\
$3$ & $(\tau h_1)^3$ & $\h$ & $\iota \tau^4 \cdot \rho^2 h_1^6$ & $(3, 3, 0)$ & $L/\rho$ \\
$5$ & $\iota v_1^4 \cdot 8 \tau^2$ & $\h$ & $\rho^2 \cdot \tau h_1 v_1^4$ & $(7, 1, 2)$ & $L/\rho$  \\
\end{longtable}

\begin{remark}
\label{rmk:L-hidden-tau^4}
The hidden extensions in Table \ref{tab:L-hidden} are $\tau^4$-periodic
in the following sense.  If we take the source and target of each
extension in $E_1(L)$ and multiply by $\tau^4$, then we obtain permanent
cycles that are related by a hidden extension.
For example, the hidden $\h$ extension from
$\tau h_1$ to $\rho h_1 \cdot \tau h_1$ generalizes to a family of 
hidden extensions
from $\tau^{4k+1} h_1$ to 
$\rho h_1 \cdot \tau^{4k+1} h_1$ for all $k \geq 0$.
\end{remark}

\begin{remark}
\label{rmk:L-hidden-v1^4}
Similarly to the $\tau^4$-periodicity discussed in 
Remark \ref{rmk:L-hidden-tau^4}, most of the hidden extensions
in Table \ref{tab:L-hidden} are $v_1^4$-periodic as well.
For example, the hidden $\h$ extension from
$\tau h_1$ to $\rho h_1 \cdot \tau h_1$ generalizes to a family of
hidden extensions from $\tau h_1 v_1^{4k}$ to 
$\rho h_1 \cdot \tau h_1 v_1^{4k}$ for all $k \geq 0$.
There are three exceptions, which appear below the horizontal
divider at the bottom of the table.  These exceptions
are discussed in more detail in Remarks \ref{rmk:L-hidden-2tau^2},
\ref{rmk:L-hidden-8tau^2v1^4i}, and
\ref{rmk:L-hidden-(tauh1)^3}.
\end{remark}

\begin{remark}
\label{rmk:L-hidden-2tau^2}
The hidden $\eta$ extension from $2 \tau^2$ to $\rho (\tau h_1)^2$
is $\tau^4$-periodic as in Remark \ref{rmk:L-hidden-tau^4},
but it is not $v_1^4$-periodic.  The elements
$2 \tau^2 v_1^{4k}$ are not permanent cycles for $k \geq 1$.
\end{remark}

\begin{remark}
\label{rmk:L-hidden-8tau^2v1^4i}
The hidden $\h$ extension from $\iota v_1^4 \cdot 8 \tau^2$
to $\rho^2 \cdot \tau h_1 v_1^4$ is $v_1^4$-periodic, but the situation
is slightly
more complicated than in Remark \ref{rmk:L-hidden-v1^4}.
For all $k$, 
$\rho^2 \cdot \tau h_1 v_1^{4k}$ receives a hidden $\h$ extension
from an appropriate multiple of
$\iota v_1^{4k} \cdot 2 \tau^2$.
For example, as shown in Figure \ref{fig:L-Einfty:1}, there is a hidden
$\h$ extension from
$\iota v_1^{4k} \cdot 16 \tau^2$ to $\rho^2 \cdot \tau h_1 v_1^8$.
\end{remark}

\begin{remark}
\label{rmk:L-hidden-(tauh1)^3}
The hidden $\h$ extension from $(\tau h_1)^3$ to 
$\iota \tau^4 \cdot \rho^2 h_1^6$ is $v_1^4$-periodic, but the situation
is more complicated than in Remarks \ref{rmk:L-hidden-v1^4} and
\ref{rmk:L-hidden-8tau^2v1^4i}.
For all $k \geq 0$, the element
$(\tau h_1)^2 \tau h_1 v_1^{4k}$ supports a hidden $\h$ extension
to the element of $E_\infty(L)$ of highest filtration in the
appropriate degree.  
For example, as shown in Figure \ref{fig:L-Einfty:7mod16},
there is a hidden $\h$ extension from
$(\tau h_1)^2 \cdot \tau^5 h_1$ to
$\iota \tau^8 \cdot \rho^3 h_1^7$.
Figures \ref{fig:L-Einfty:3mod8}, 
\ref{fig:L-Einfty:7mod16}, and \ref{fig:L-Einfty:15mod32}
show several extensions of this type.
\end{remark}

\section{Charts}
\label{sec:charts}

We explain the notation used in the charts.
\begin{itemize} 
\item 
The horizontal coordinate is the stem $s$.
The vertical coordinate is the Adams-Novikov filtration $f$
(see Section \ref{subsec:notation} for further discussion).
\item
Black or green circles represent copies of $\F_2$, periodicized 
by some power of $\tau$.  The relevant power of $\tau$
varies from chart to chart.
\item
Black or green unfilled boxes represent copies of $\Z$ (the $2$-adic integers),
periodicized by some power of $\tau$.
The relevant power of $\tau$ varies from chart to chart.
\item
Black or green boxes containing a number $n$ represent
copies of $\Z/2^n$,
periodicized by some power of $\tau$.
The relevant power of $\tau$ varies from chart to chart.
\item
Red unfilled boxes represent copies of $\Z$ (the $2$-adic integers)
that are not $\tau^k$-periodic for any $k$.
\item
Green objects represent elements in the image of the map
$\Sigma^{-1} \ko \to L$ (or $\Sigma^{-1} \ko^\C \to L^\C$).
\item
Black objects represent elements in the cokernel of the map
$\Sigma^{-1} \ko \to L$ (or $\Sigma^{-1} \ko^\C \to L^\C$).
In other words, they are detected by the map $L \to \ko$
(or $L^\C \to \ko^\C$).
\item
Lines of slope $1$ represent $h_1$-multiplications.
\item
Black or green arrows of slope $1$ represent infinite sequences of elements
that are $\tau^k$-periodic for some $k > 0$ and are
connected by $h_1$-multiplications.
\item
Red arrows of slope $1$ represent infinite sequences of elements
that are connected by $h_1$-multiplications and 
are not $\tau^k$-periodic for any $k$.
\item
Lines of slope $-1$ represent $\rho$-multiplications.
\item
Dashed lines of slope $-1$ represent $\rho$-multiplications
whose values are multiples of $\tau^k$ for some $k > 0$.
For example, in Figure \ref{fig:ko-Einfty:0},
we have $\rho \cdot \rho^3 v_1^4$ equals
$\tau^4 \cdot h_1^4$.
\item
Black or green arrows of slope $-1$ represent infinite sequences of elements
that are $\tau^k$-periodic for some $k > 0$ and are
connected by $\rho$-multiplications.
\item
Light blue lines of slope $-3$ represent effective $d_1$ differentials.
\item
Dashed light blue lines of slope $-3$ 
represent effective $d_1$ differentials that hit multiples of $\tau^k$,
for some $k > 0$.
For example, the dashed line in Figure \ref{fig:ko-C-E1}
indicates that $d_1(v_1^2)$ equals $\tau h_1^3$.
\item
Dark blue lines indicate hidden extensions by
$\h$, $\rho$, or $h_1$.  
\item
Dashed dark blue lines indicate hidden extensions whose
value is a multiple of $\tau^k$ for some $k > 0$.
For example, in Figure \ref{fig:L-C-Einfty}, 
there is a hidden $\h$ extension from $\iota 4 v_1^2$
to $\tau h_1^3$.
\end{itemize}

\begin{figure}[H]
\begin{center}
\makebox[\textwidth][c]{\includegraphics[trim={0cm, 2.0cm, 20pt, 20pt}, clip, page=1]{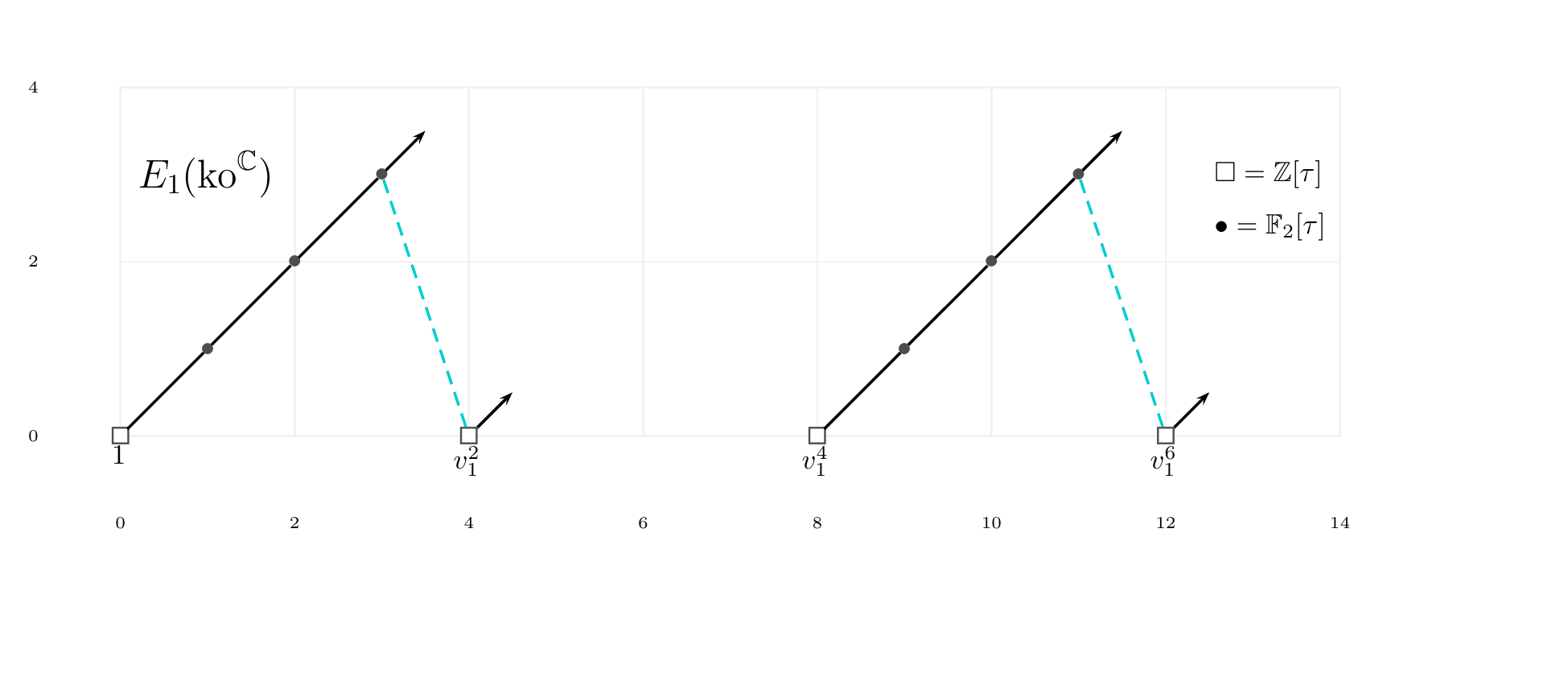}}
\caption{The $E_1$-page of the effective spectral sequence for
$\ko^\C$}
\label{fig:ko-C-E1}
\hfill
\end{center}
\end{figure}

\begin{figure}[H]
\begin{center}
\makebox[\textwidth][c]{\includegraphics[trim={0cm, 2.0cm, 20pt, 20pt}, clip, page=1]{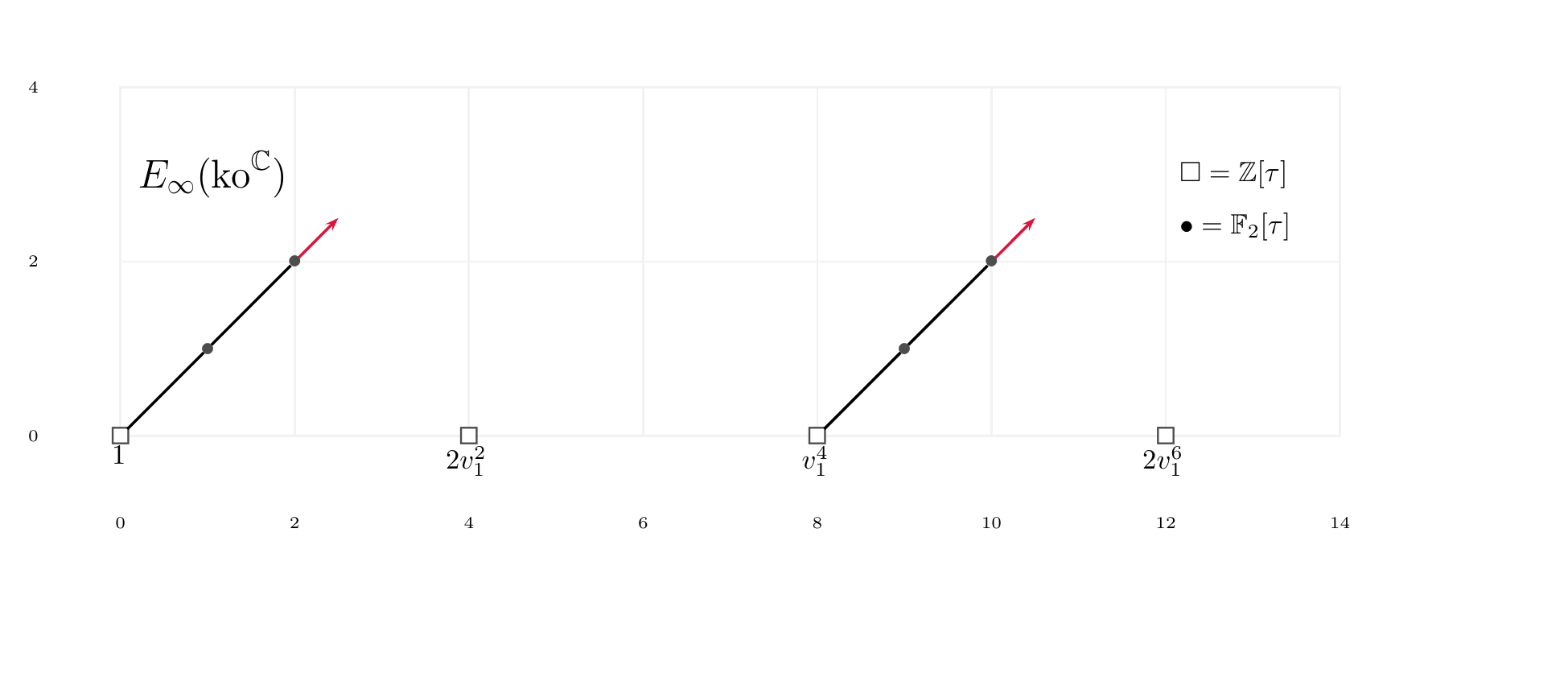}}
\caption{The $E_\infty$-page of the effective spectral sequence for
$\ko^\C$}
\label{fig:ko-C-Einfty}
\hfill
\end{center}
\end{figure}

\begin{figure}[H]
\begin{center}
\makebox[\textwidth][c]{\includegraphics[trim={0cm, 2.0cm, 20pt, 20pt}, clip, page=1]{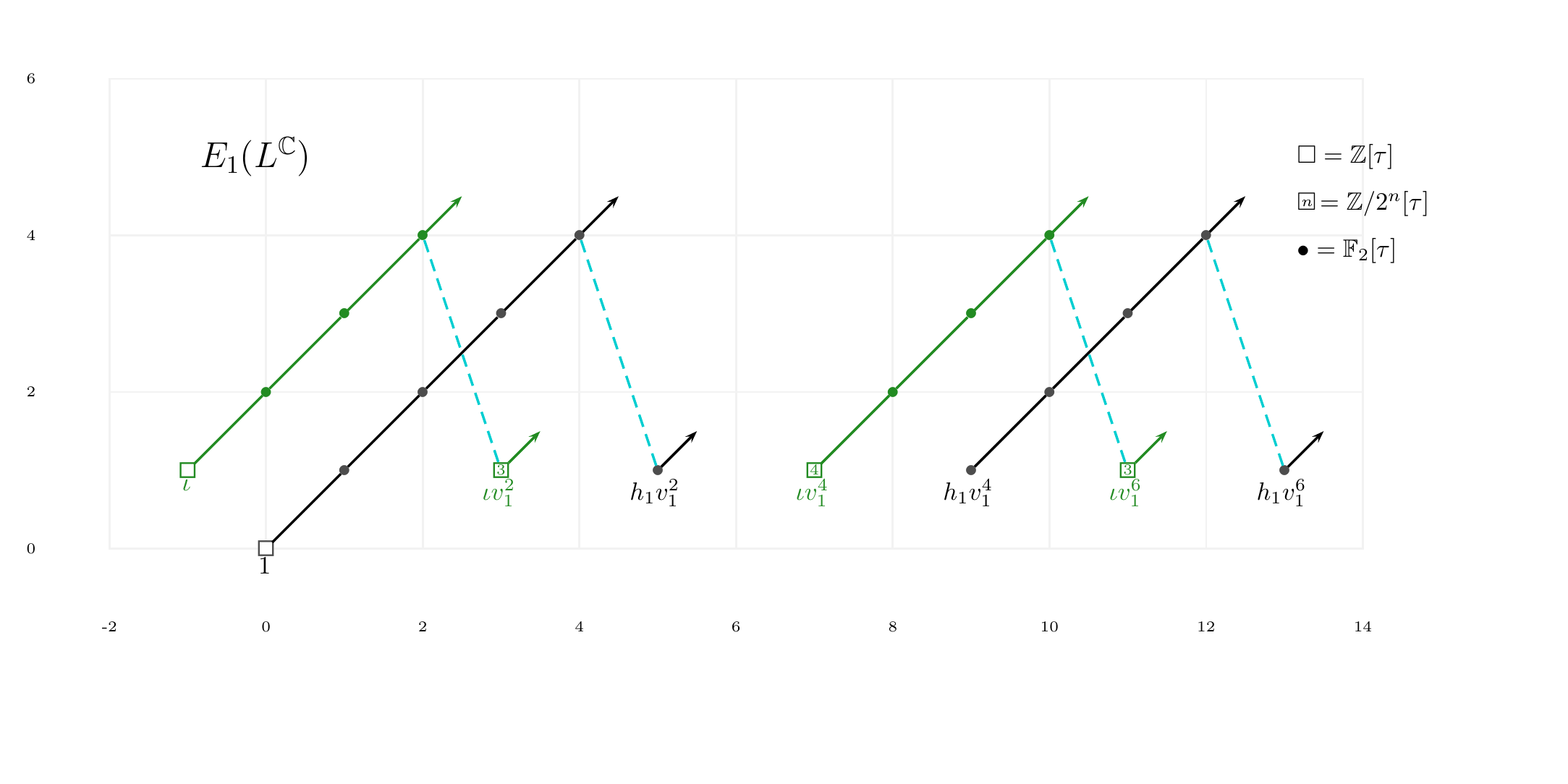}}
\caption{The $E_1$-page of the effective spectral sequence for
$L^\C$}
\label{fig:L-C-E1}
\hfill
\end{center}
\end{figure}

\begin{figure}[H]
\begin{center}
\makebox[\textwidth][c]{\includegraphics[trim={0cm, 2.0cm, 20pt, 20pt}, clip, page=1]{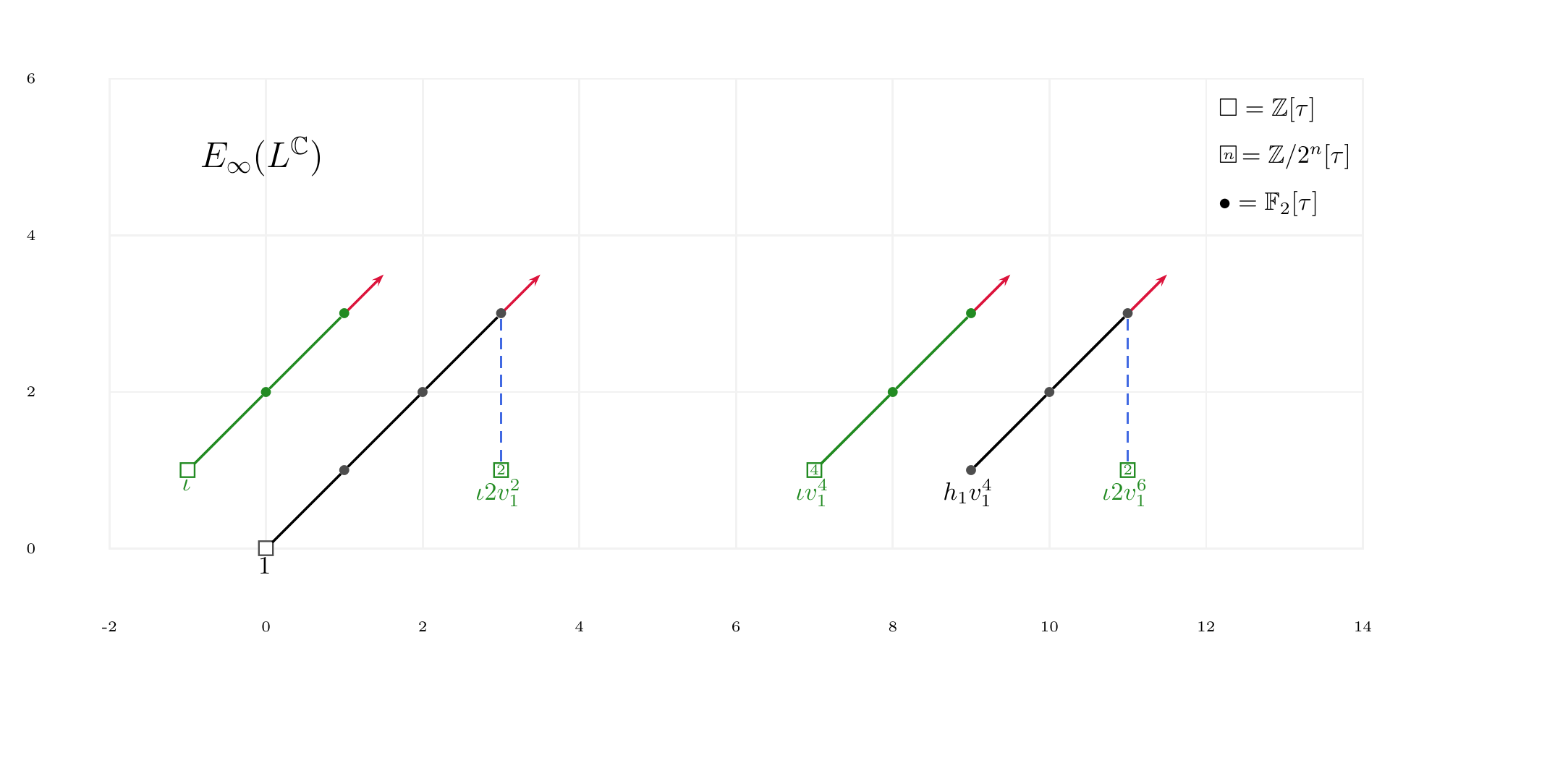}}
\caption{The $E_\infty$-page of the effective spectral sequence for
$L^\C$}
\label{fig:L-C-Einfty}
\hfill
\end{center}
\end{figure}

\begin{figure}[H]
\begin{center}
\makebox[\textwidth][c]{\includegraphics[trim={0cm, 2.0cm, 20pt, 20pt}, clip, page=1]{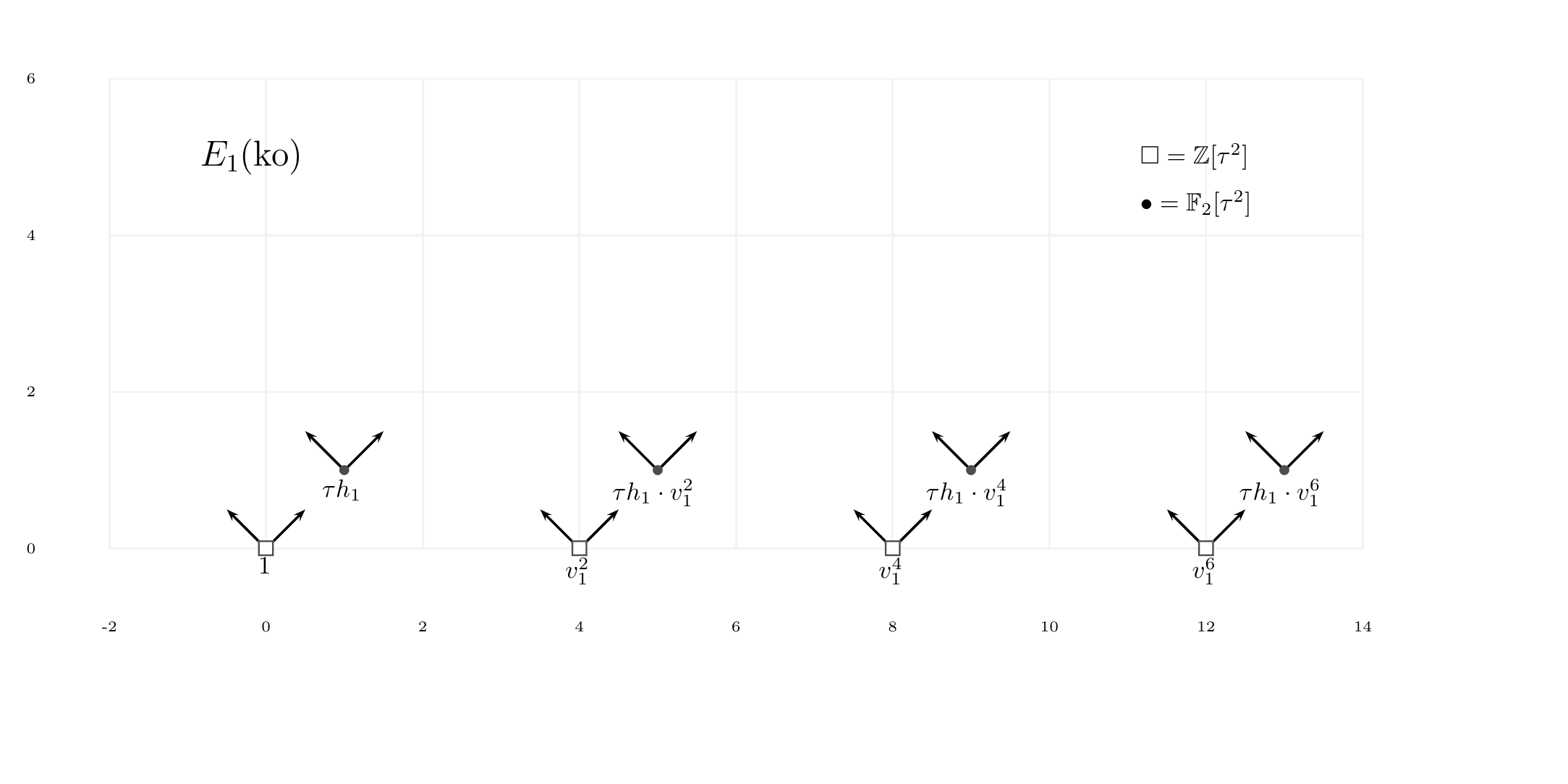}}
\caption{The $E_1$-page of the effective spectral sequence for
$\ko$}
\label{fig:ko-E1}
\hfill
\end{center}
\end{figure}

\begin{figure}[H]
\begin{center}
\makebox[\textwidth][c]{\includegraphics[trim={0cm, 2.0cm, 20pt, 20pt}, clip, page=1]{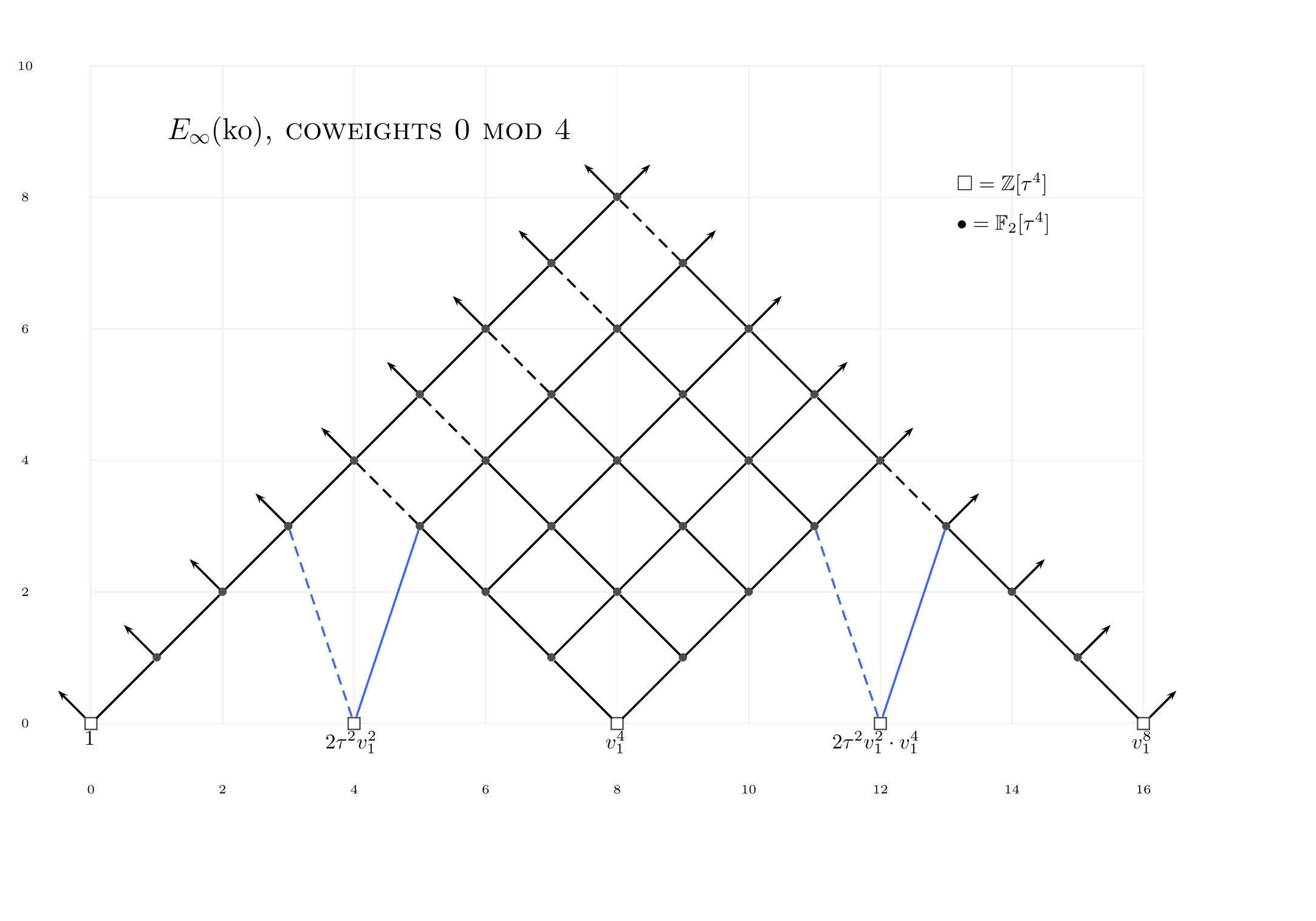}}
\caption{The $E_\infty$-page of the effective spectral sequence for
$\ko$ in coweights $0$ mod $4$}
\label{fig:ko-Einfty:0}
\hfill
\end{center}
\end{figure}

\KOMAoption{paper}{landscape,10.4in:11.0in}
\KOMAoption{DIV}{last}
\newgeometry{margin=0in,top=0.5in,footskip=0.3in}

\begin{figure}[H]
\begin{center}
\makebox[\textwidth][c]{\includegraphics[trim={0cm, 2.0cm, 20pt, 20pt}, clip, page=1]{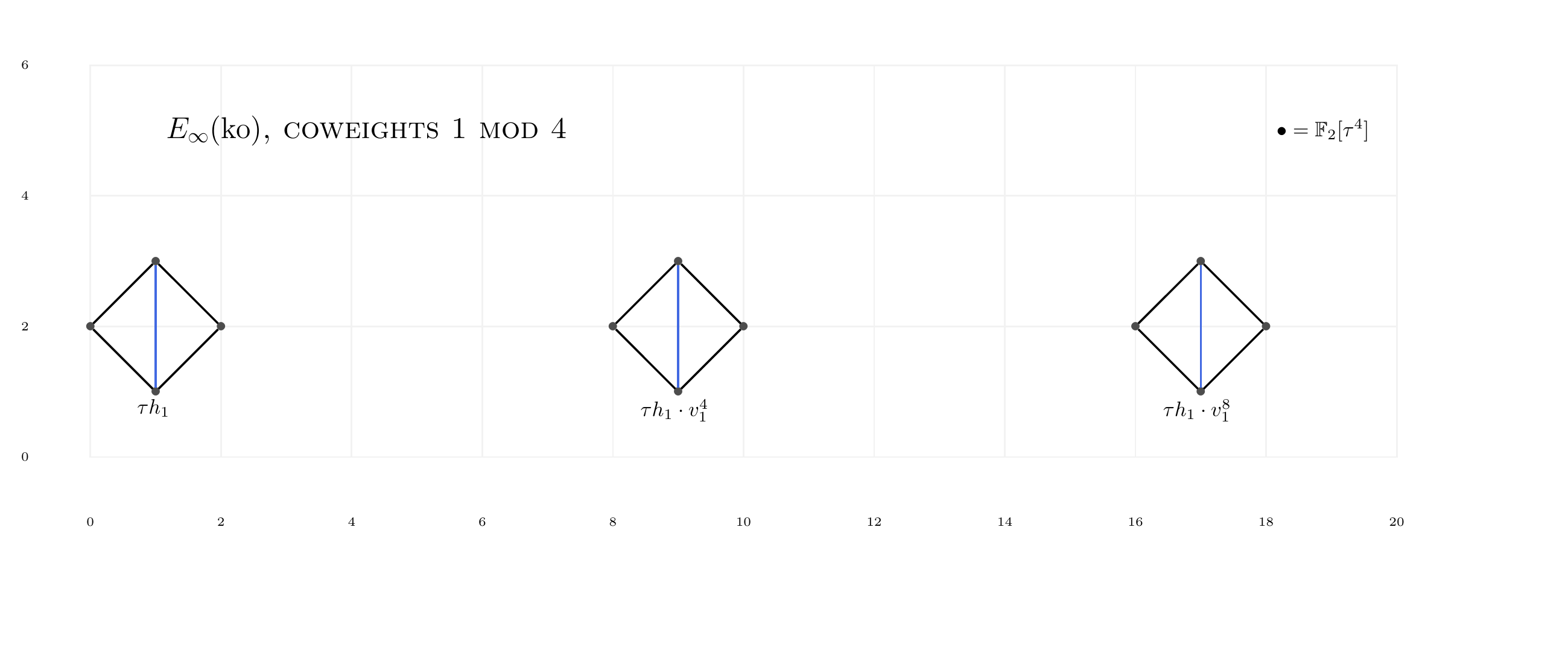}}
\caption{The $E_\infty$-page of the effective spectral sequence for
$\ko$ in coweights $1$ mod $4$}
\label{fig:ko-Einfty:1}
\hfill
\end{center}
\end{figure}

\begin{figure}[H]
\begin{center}
\makebox[\textwidth][c]{\includegraphics[trim={0cm, 2.0cm, 20pt, 20pt}, clip, page=1]{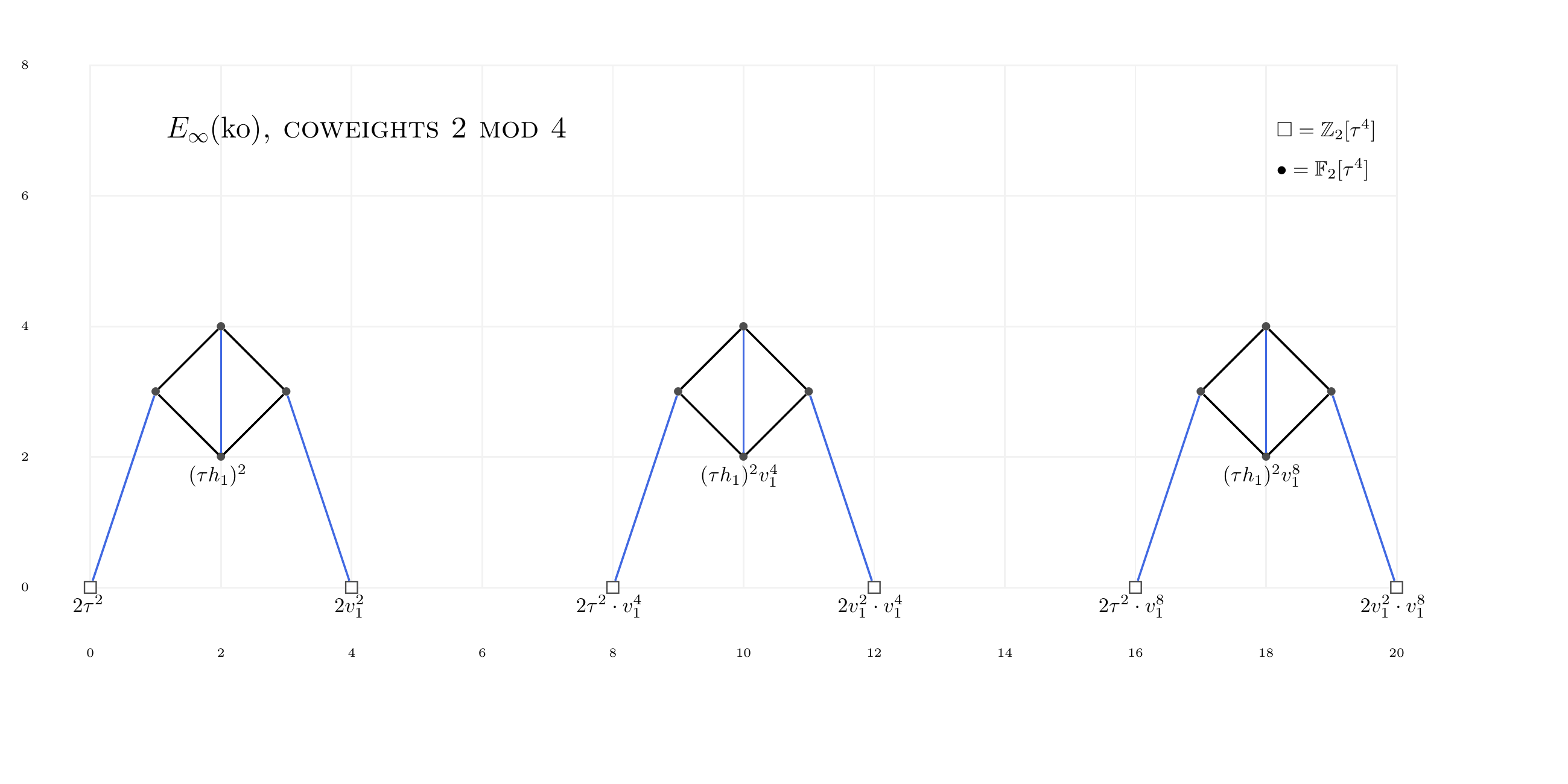}}
\caption{The $E_\infty$-page of the effective spectral sequence for
$\ko$ in coweights $2$ mod $4$}
\label{fig:ko-Einfty:2}
\hfill
\end{center}
\end{figure}

\newpage

\newgeometry{margin=0.5in,footskip=0.3in}

\begin{figure}[H]
\begin{center}
{\includegraphics[trim={0cm, 2.0cm, 20pt, 20pt}, clip, page=1]{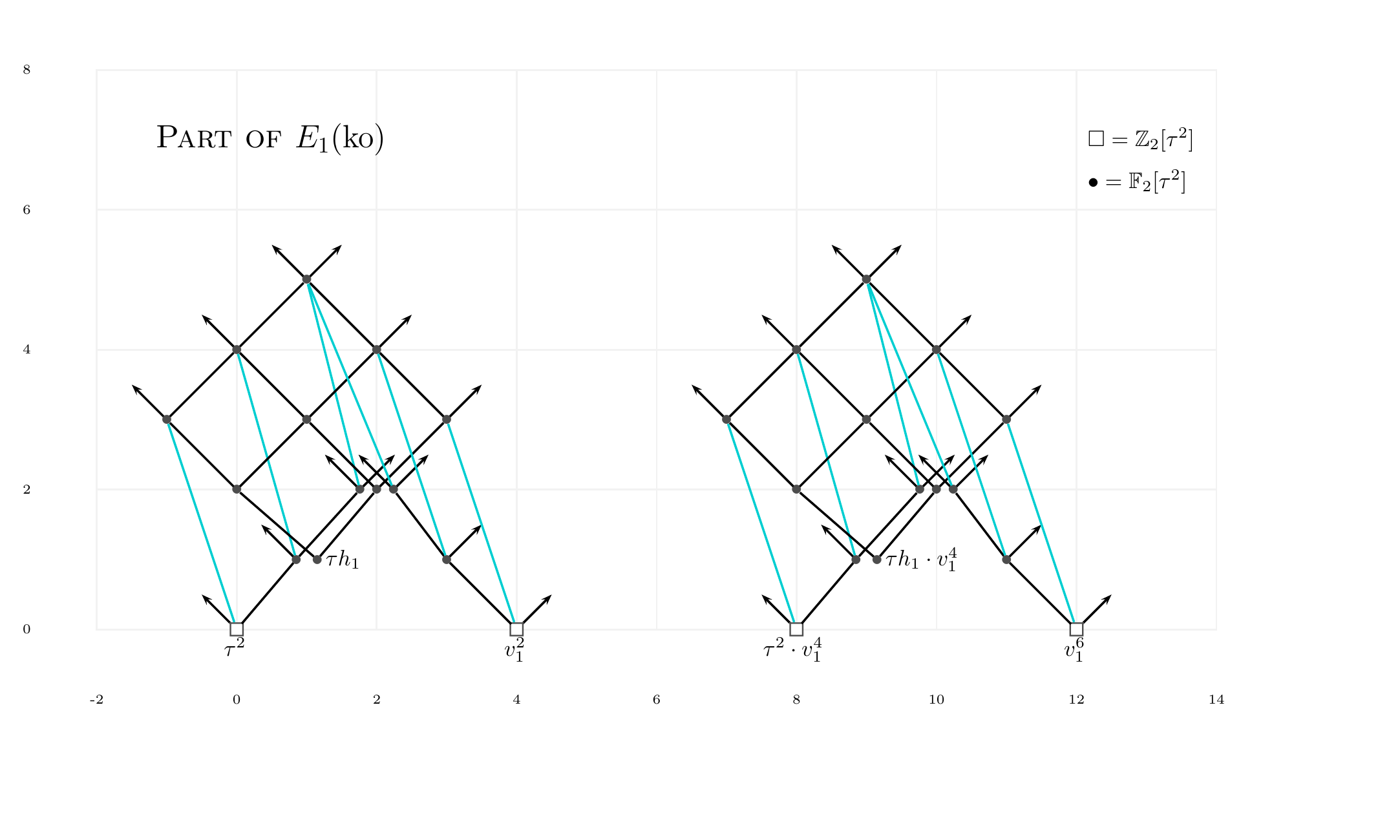}}
\caption{Some differentials in the effective spectral sequence
for $\ko$}
\label{fig:ko-d1:2-1}
\hfill
\end{center}
\end{figure}

\begin{figure}[H]
\begin{center}
{\includegraphics[trim={10pt, 2.0cm, 20pt, 20pt}, clip, page=1]{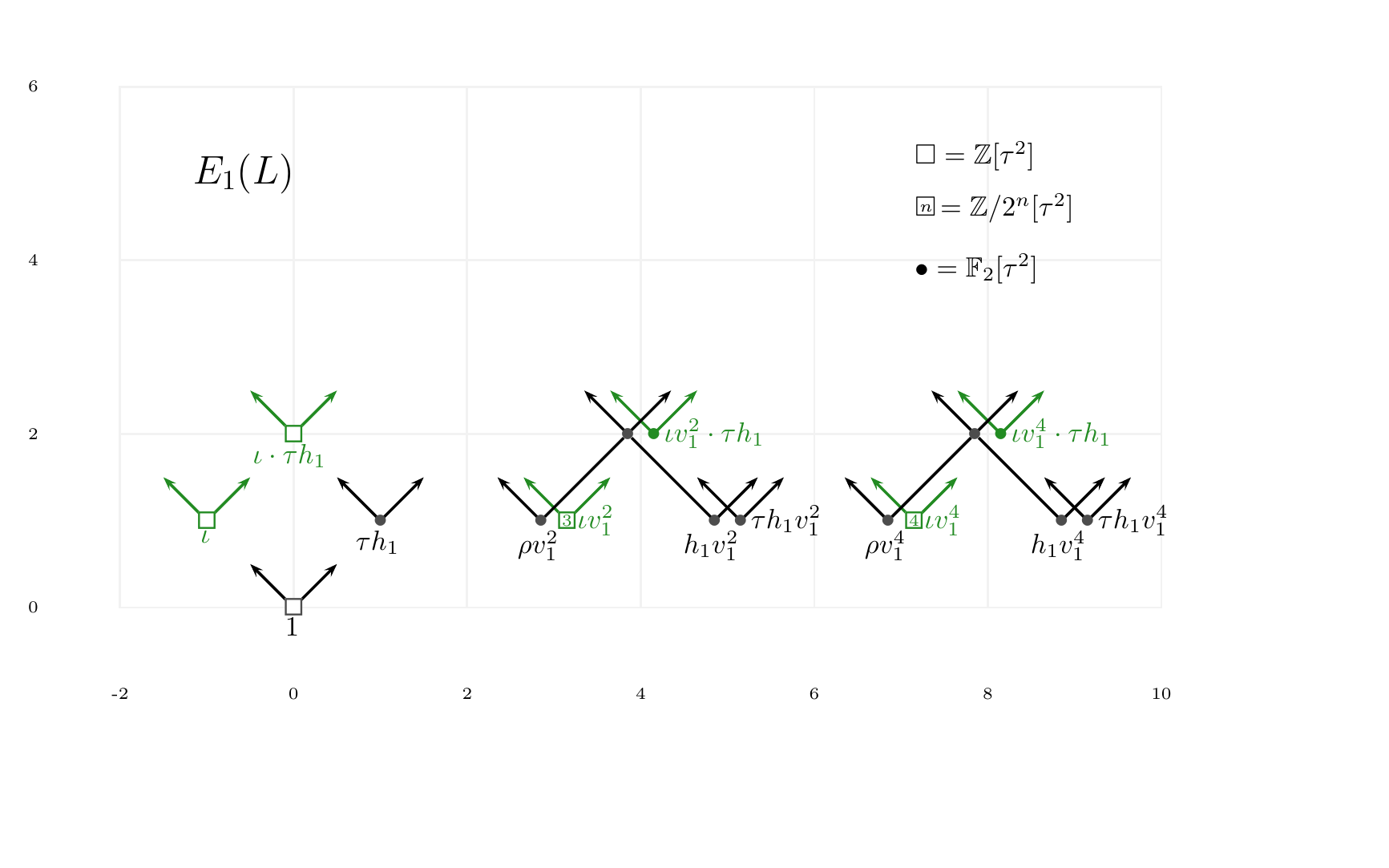}}
\caption{The $E_1$-page of the effective spectral sequence for $L$}
\label{fig:L-E1}
\hfill
\end{center}
\end{figure}

\KOMAoption{paper}{landscape,10.4in:7.0in}
\KOMAoption{DIV}{last}
\newgeometry{margin=0in,footskip=0.3in}

\begin{figure}[H]
\begin{center}
{\includegraphics[trim={0cm, 2.0cm, 20pt, 20pt}, clip, page=1]{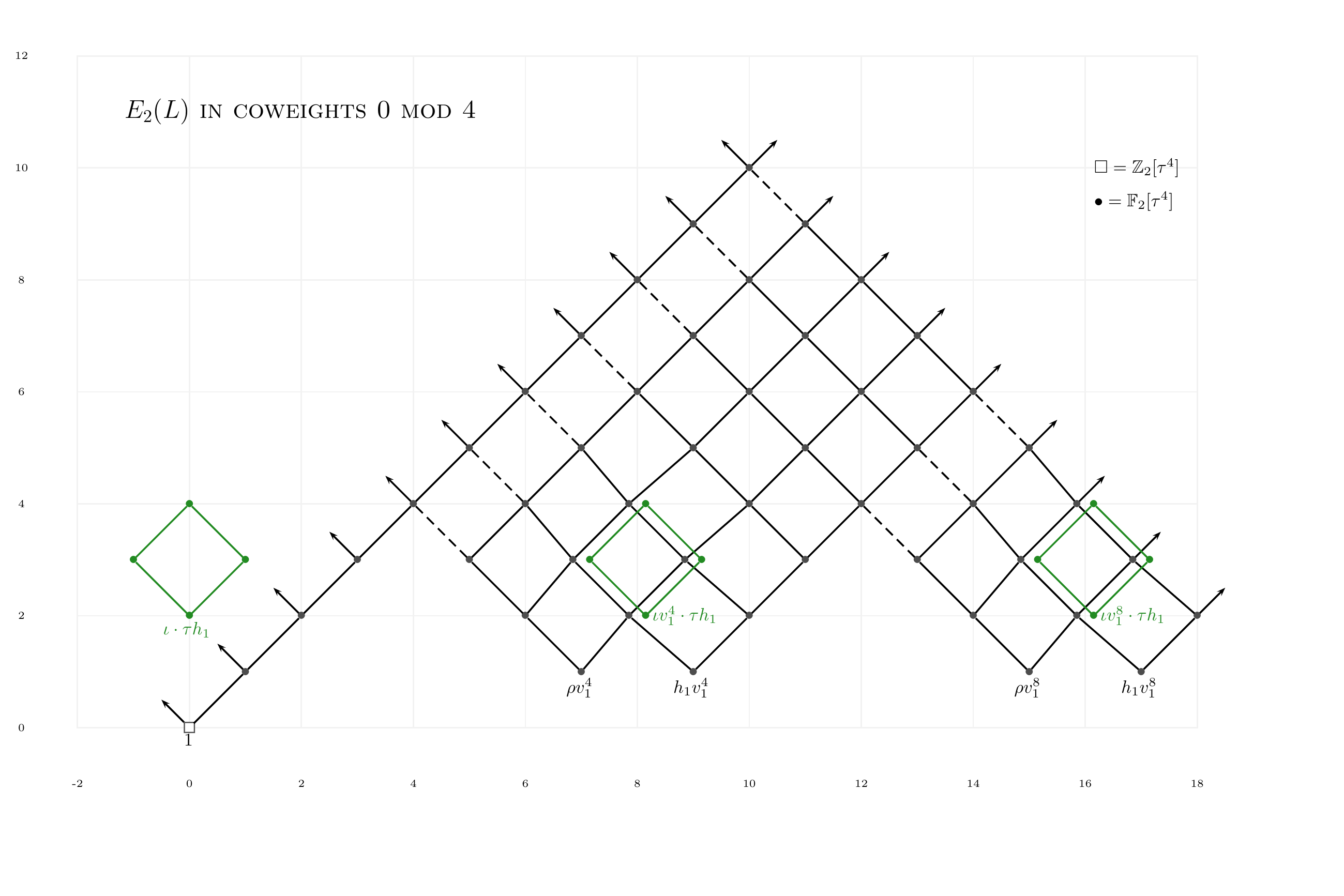}}
\caption{The $E_2$-page of the effective
spectral sequence for $L$ in coweights $0$ mod $4$}
\label{fig:L-E2:0}
\hfill
\end{center}
\end{figure}

\KOMAoption{paper}{landscape,9.6in:7.0in}
\KOMAoption{DIV}{last}
\newgeometry{margin=0in,footskip=0.3in}

\begin{figure}[H]
\begin{center}
{\includegraphics[trim={0cm, 2.0cm, 20pt, 20pt}, clip, page=1]{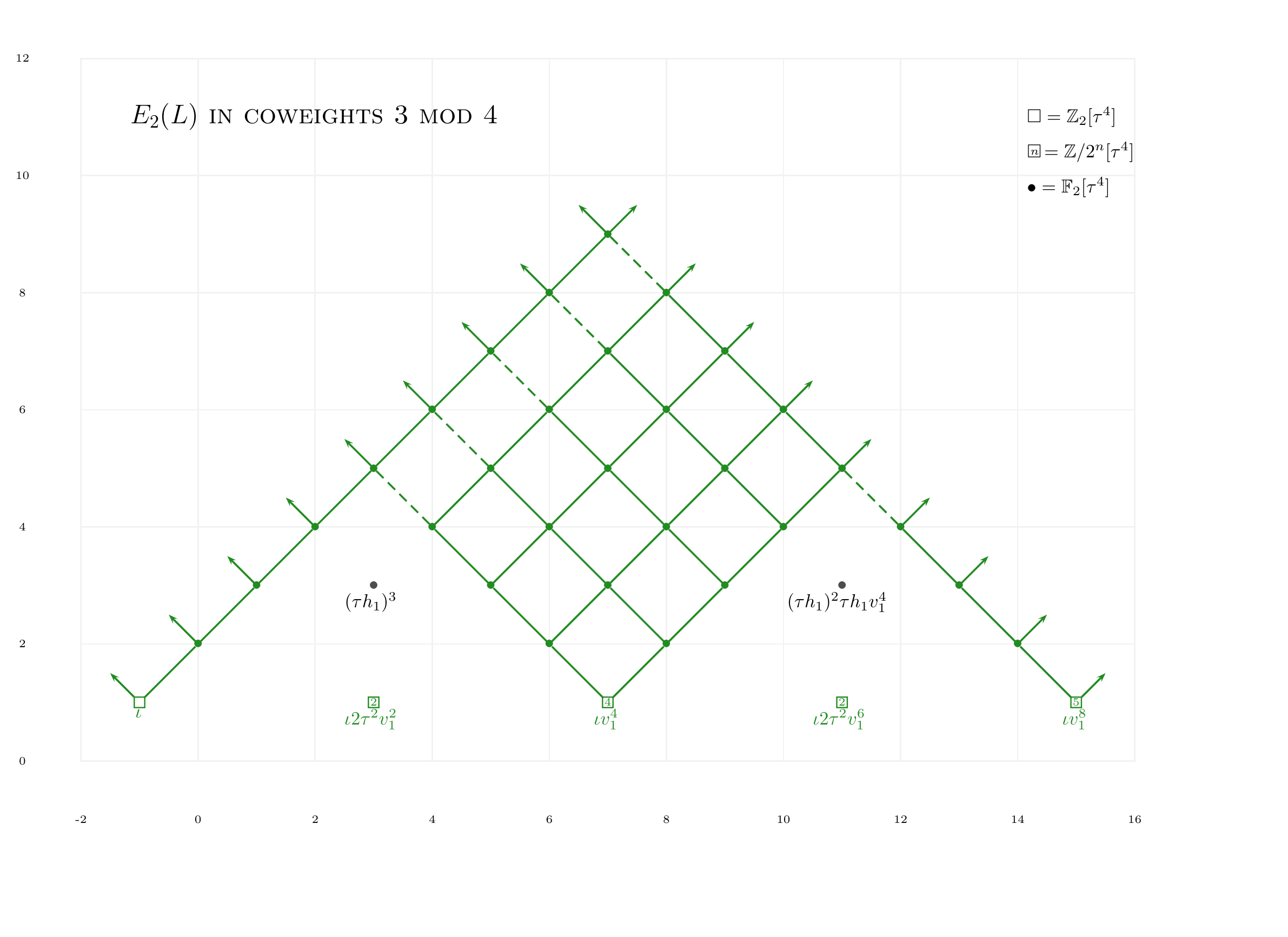}}
\caption{The $E_2$-page of the effective
spectral sequence for $L$ in coweights $3$ mod $4$}
\label{fig:L-E2:3}
\hfill
\end{center}
\end{figure}

\KOMAoption{paper}{landscape,11.3in:10.2in}
\KOMAoption{DIV}{last}
\newgeometry{margin=0in,footskip=0.3in}

\begin{figure}[H]
\begin{center}
{\includegraphics[trim={0cm, 0.0cm, 0pt, -10pt}, clip, page=1]{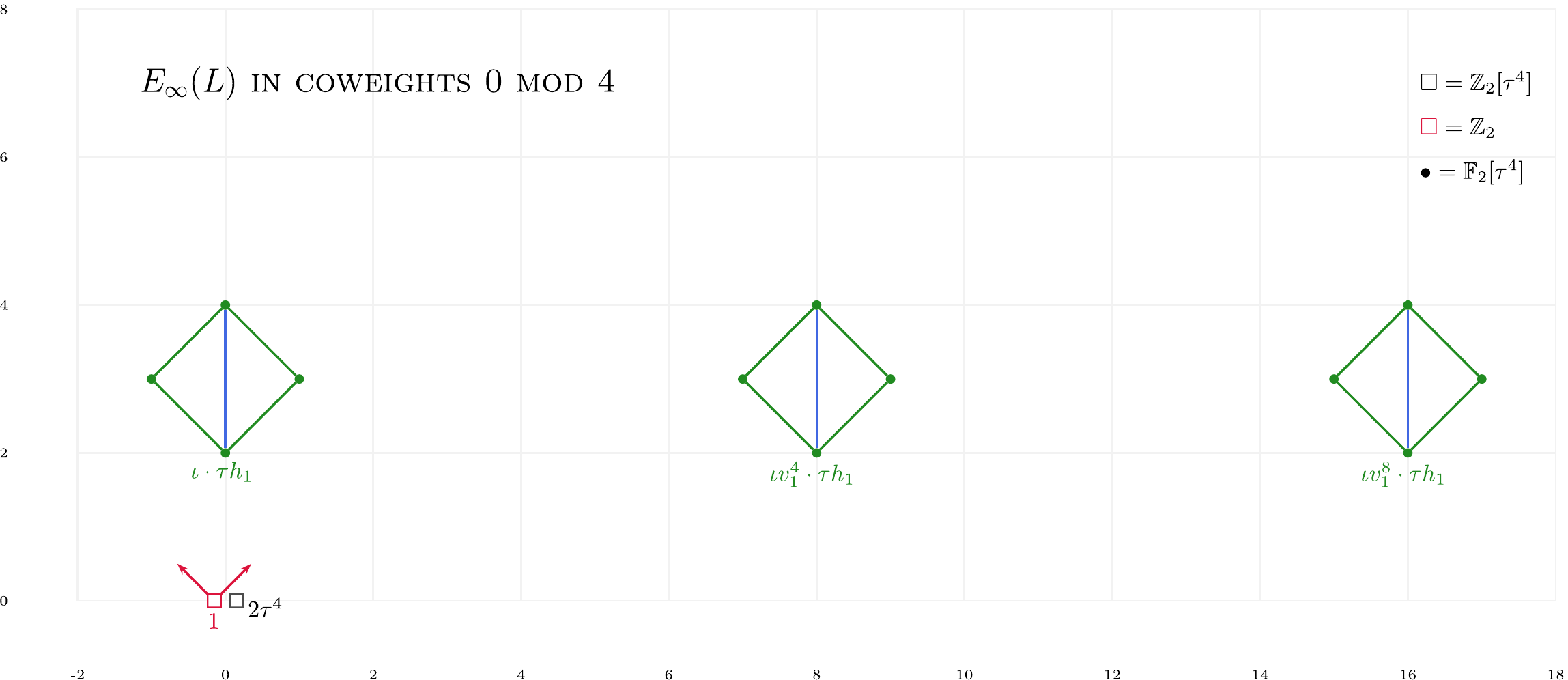}}
\caption{The $E_\infty$-page of the effective
spectral sequence for $L$ in coweights $0$ mod $4$}
\label{fig:L-Einfty:0}
\hfill
\end{center}
\end{figure}

\begin{figure}[H]
\begin{center}
{\includegraphics[trim={0cm, 2.0cm, 20pt, 20pt}, clip, page=1]{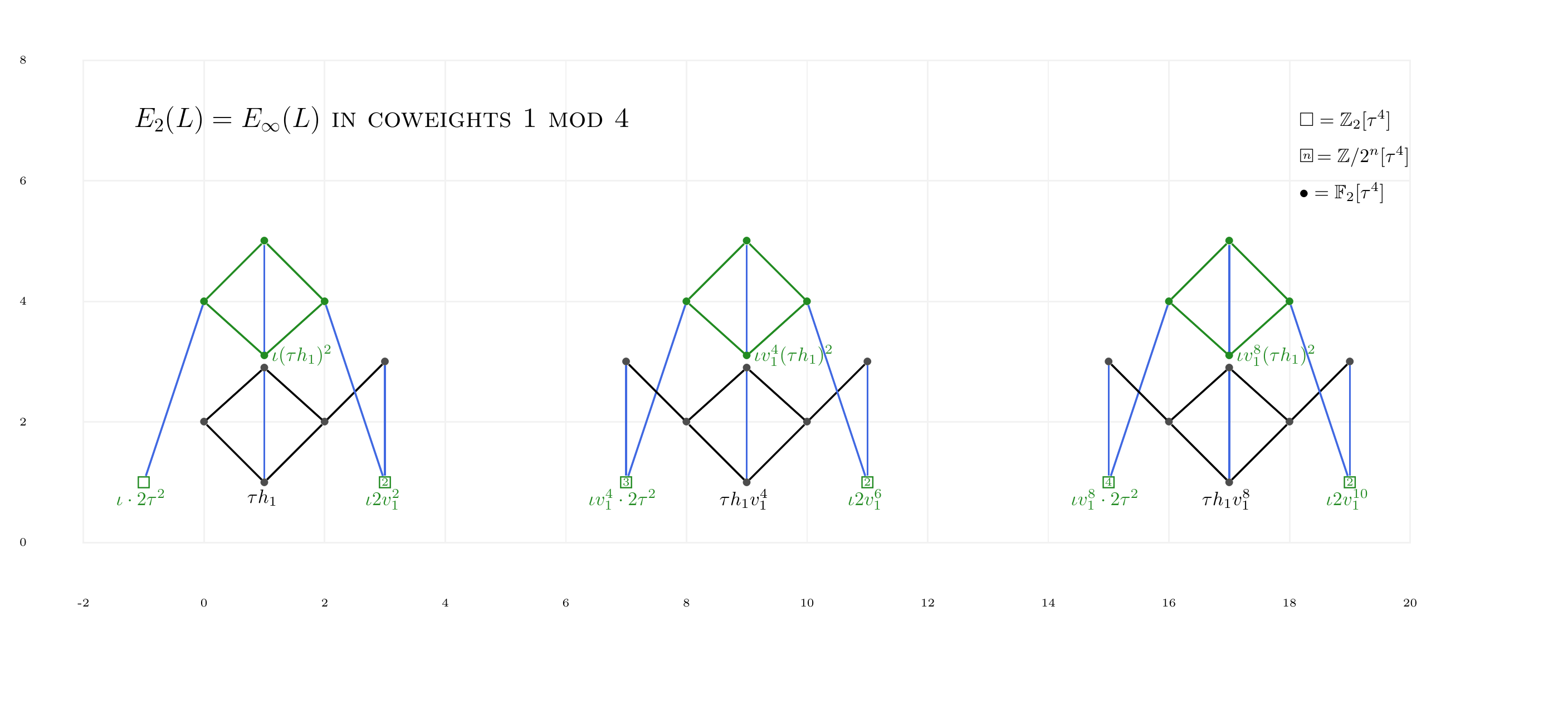}}
\caption{The $E_\infty$-page of the effective
spectral sequence for $L$ in coweights $1$ mod $4$}
\label{fig:L-Einfty:1}
\hfill
\end{center}
\end{figure}

\KOMAoption{paper}{landscape,10.4in:8.2in}
\KOMAoption{DIV}{last}
\newgeometry{margin=0in,footskip=0.3in}

\begin{figure}[H]
\begin{center}
{\includegraphics[trim={0cm, 2.0cm, 20pt, 20pt}, clip, page=1]{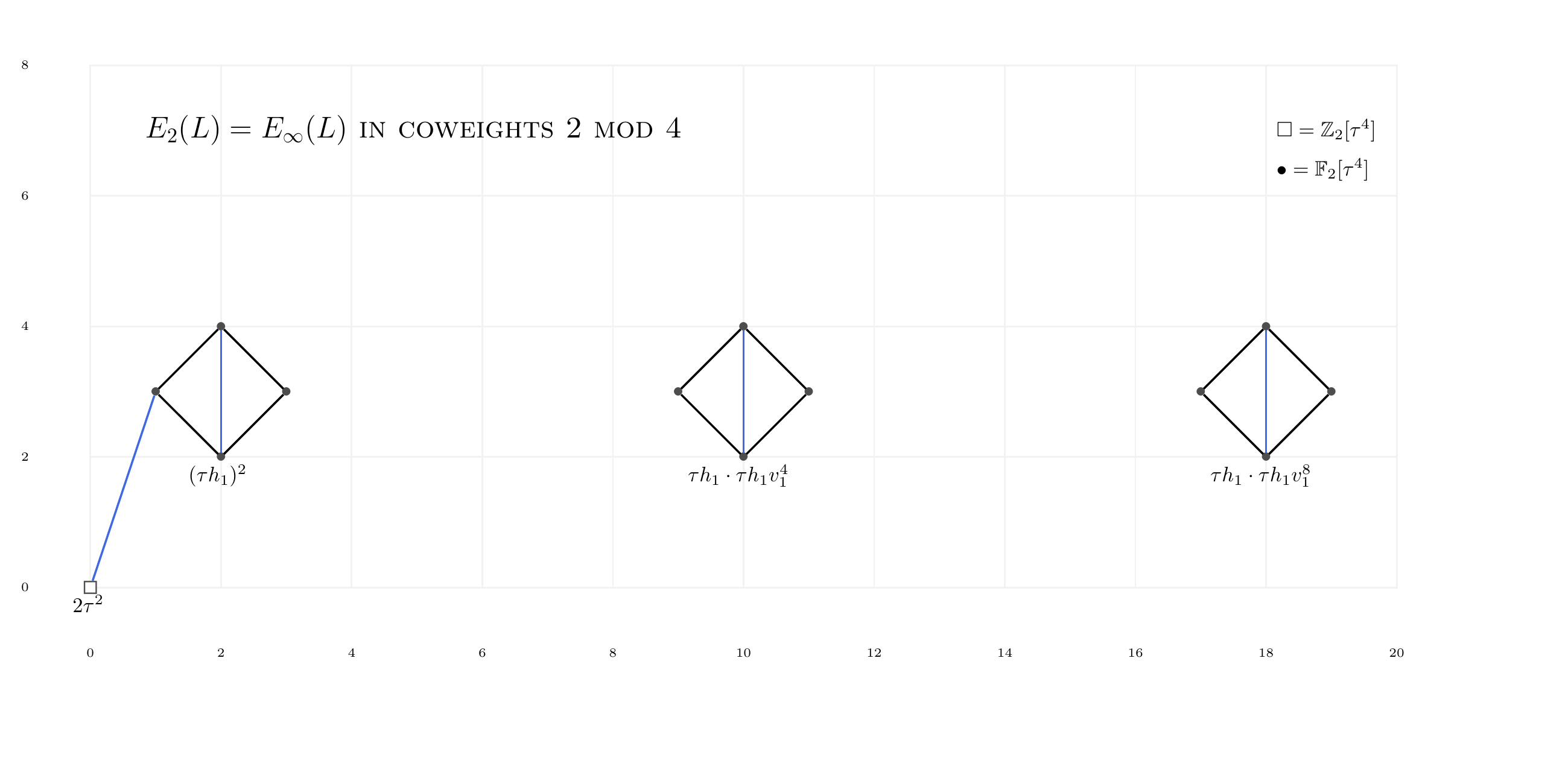}}
\caption{The $E_\infty$-page of the effective
spectral sequence for $L$ in coweights $2$ mod $4$}
\label{fig:L-Einfty:2}
\hfill
\end{center}
\end{figure}

\begin{figure}[H]
\begin{center}
{\includegraphics[trim={0cm, 0pt, 20pt, 20pt}, clip, page=1]{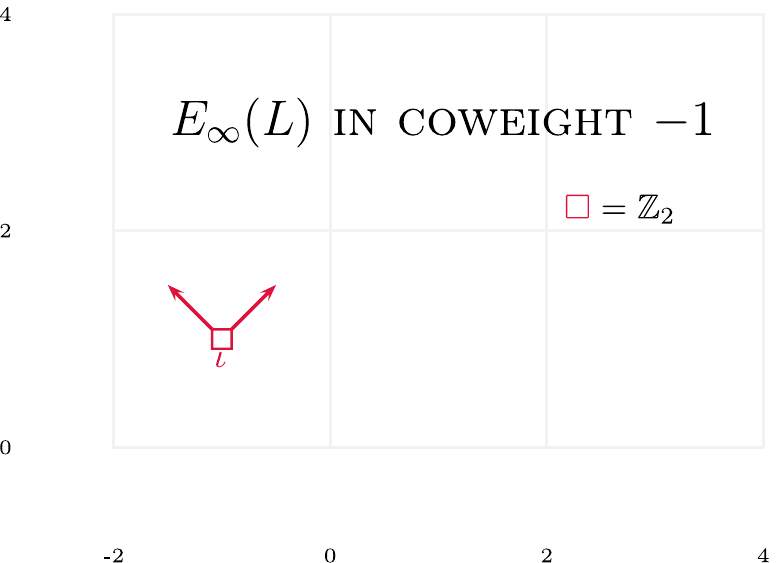}}
\caption{The $E_\infty$-page of the effective
spectral sequence for $L$ in coweight $-1$}
\label{fig:L-Einfty:-1}
\hfill
\end{center}
\end{figure}

\KOMAoption{paper}{landscape,20in:7.8in}
\KOMAoption{DIV}{last}
\newgeometry{margin=0in,footskip=0.3in}

\begin{figure}[H]
\begin{center}
{\includegraphics[trim={0cm, 2.0cm, 20pt, 20pt}, clip, page=1]{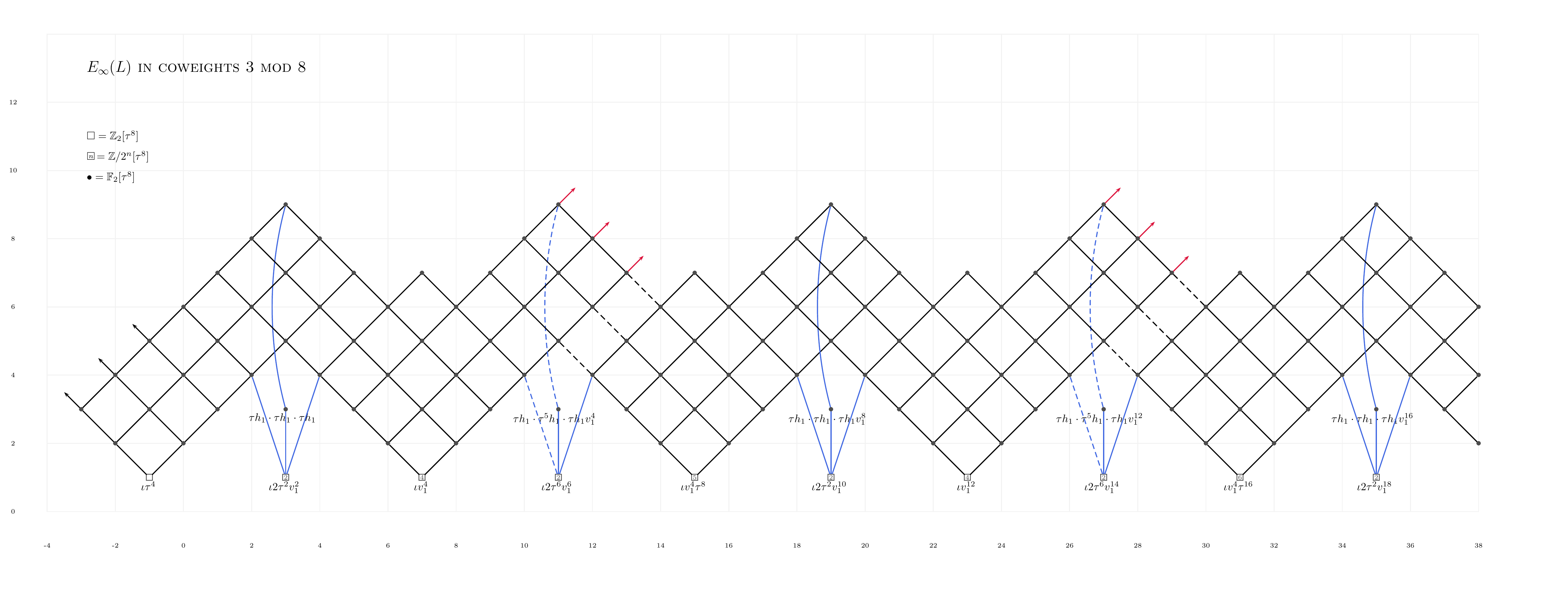}}
\caption{The $E_\infty$-page of the effective
spectral sequence for $L$ in coweights $3$ mod $8$}
\label{fig:L-Einfty:3mod8}
\hfill
\end{center}
\end{figure}

\KOMAoption{paper}{landscape,33.78in:8.66in}
\KOMAoption{DIV}{last}
\newgeometry{margin=0in,footskip=0.3in}

\begin{figure}[H]
\begin{center}
{\includegraphics[trim={0cm, 2.0cm, 20pt, 20pt}, clip, page=1]{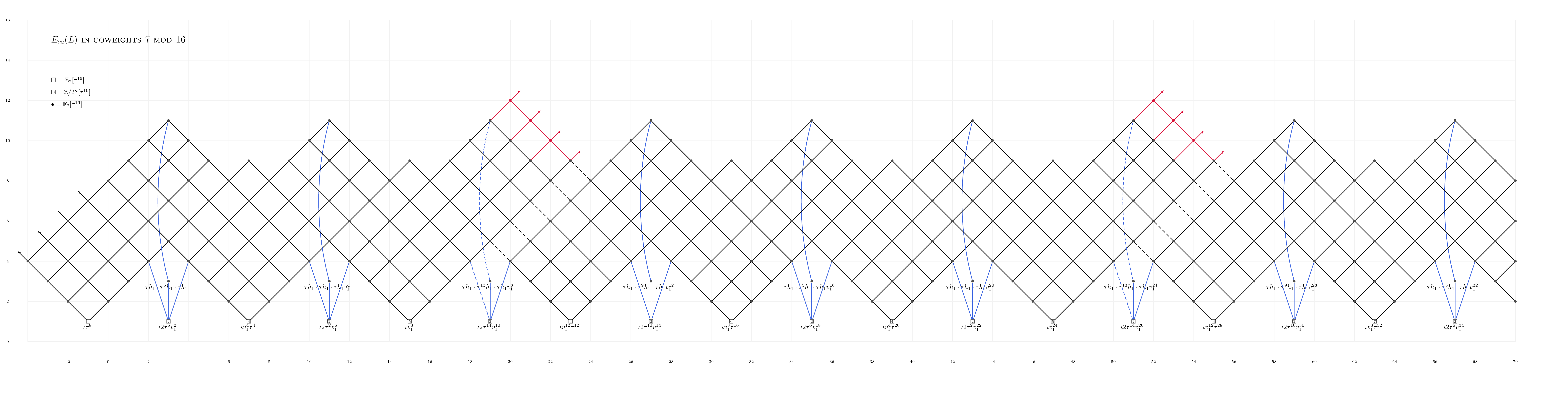}}
\caption{The $E_\infty$-page of the effective
spectral sequence for $L$ in coweights $7$ mod $16$}
\label{fig:L-Einfty:7mod16}
\hfill
\end{center}
\end{figure}

\KOMAoption{paper}{landscape,23.4in:10.4in}
\KOMAoption{DIV}{last}
\newgeometry{margin=0in,footskip=0.0in}

\begin{figure}[H]
\begin{center}
{\includegraphics[trim={0cm, 2.0cm, 20pt, 20pt}, clip, page=1]{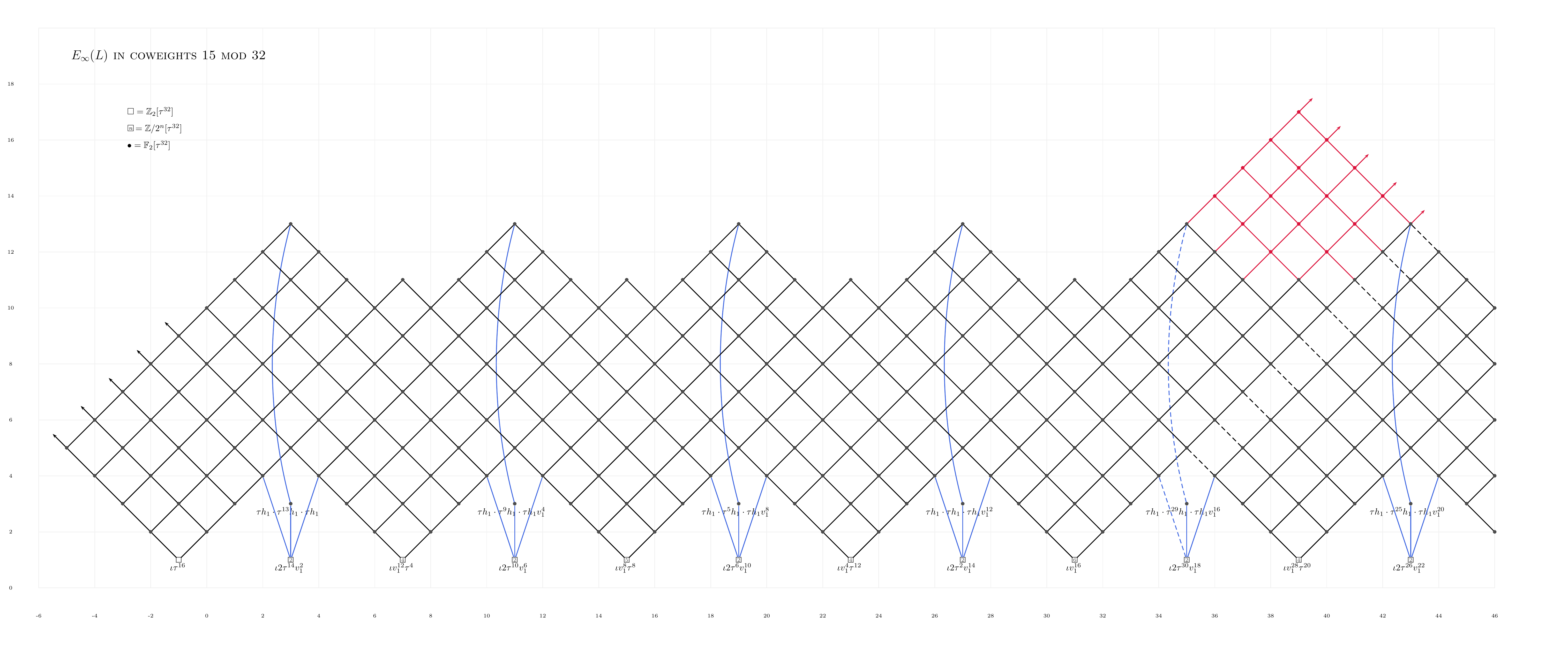}}
\caption{The $E_\infty$-page of the effective
spectral sequence for $L$ in coweights $15$ mod $32$}
\label{fig:L-Einfty:15mod32}
\hfill
\end{center}
\end{figure}

\KOMAoption{paper}{portrait,8.5in:11in}
\KOMAoption{DIV}{last}
\newgeometry{margin=1in,footskip=0.5in}

\bibliographystyle{alpha}
\bibliography{slice-v1.bib}

\end{document}